\documentclass[reqno,11pt]{amsart}
\usepackage{amsmath, latexsym, amsfonts, amssymb, amsthm, amscd}
\usepackage{graphics,epsf,psfrag}
\usepackage{graphics,epsf,psfrag}
 \usepackage[utf8]{inputenc}
\usepackage{stmaryrd}
\usepackage{mathabx}
\usepackage{hyperref}
\usepackage{xcolor}

\renewcommand{\bar}{\overline}

\newcommand{\lint}{\llbracket}
\newcommand{\rint}{\rrbracket}

\numberwithin{equation}{section}

\newtheorem{theorema}{Theorem}

\newtheorem{theorem}{Theorem}[section]
\newtheorem{lemma}[theorem]{Lemma}
\newtheorem{proposition}[theorem]{Proposition}

\newtheorem{rem}[theorem]{Remark}

\newcommand{\red}{}

\newcommand{\dd}{\mathrm{d}}
\newcommand{\ind}{\mathbf{1}}

\newcommand{\Supp}{\mathrm{Supp}}

\renewcommand{\tilde}{\widetilde}
\renewcommand{\hat}{\widehat}

\newcommand{\cG}{{\ensuremath{\mathcal G}} }

\newcommand{\cA}{{\ensuremath{\mathcal A}} }
\newcommand{\cB}{{\ensuremath{\mathcal B}} }
\newcommand{\cF}{{\ensuremath{\mathcal F}} }

\newcommand{\cP}{{\ensuremath{\mathcal P}} }

\newcommand{\cC}{{\ensuremath{\mathcal C}} }

\newcommand{\cD}{{\ensuremath{\mathcal D}} }

\newcommand{\cM}{{\ensuremath{\mathcal M}} }

\DeclareMathSymbol{\leqslant}{\mathalpha}{AMSa}{"36} 
\DeclareMathSymbol{\geqslant}{\mathalpha}{AMSa}{"3E} 
\DeclareMathSymbol{\eset}{\mathalpha}{AMSb}{"3F}     

\newcommand{\Var}{\mathrm{Var}}        
\newcommand{\maxtwo}[2]{\max_{\substack{#1 \\ #2}}} 
\newcommand{\suptwo}[2]{\sup_{\substack{#1 \\ #2}}} 
\newcommand{\sumtwo}[2]{\sum_{\substack{#1 \\ #2}}} 

\newcommand{\bbC}{{\ensuremath{\mathbb C}} }
\newcommand{\bbD}{{\ensuremath{\mathbb D}} }
\newcommand{\bbE}{{\ensuremath{\mathbb E}} }

\newcommand{\bbL}{{\ensuremath{\mathbb L}} }

\newcommand{\bbP}{{\ensuremath{\mathbb P}} }

\newcommand{\bbR}{{\ensuremath{\mathbb R}} }

\newcommand{\gep}{\varepsilon}       

\newcommand{\gO}{\Omega}
\newcommand{\gl}{\lambda}

\makeatletter
\def\captionfont@{\footnotesize}
\def\captionheadfont@{\scshape}

\long\def\@makecaption#1#2{%
  \vspace{2mm}
  \setbox\@tempboxa\vbox{\color@setgroup
    \advance\hsize-6pc\noindent
    \captionfont@\captionheadfont@#1\@xp\@ifnotempty\@xp
        {\@cdr#2\@nil}{.\captionfont@\upshape\enspace#2}%
    \unskip\kern-6pc\par
    \global\setbox\@ne\lastbox\color@endgroup}%
  \ifhbox\@ne 
    \setbox\@ne\hbox{\unhbox\@ne\unskip\unskip\unpenalty\unkern}%
  \fi
  \ifdim\wd\@tempboxa=\z@ 
    \setbox\@ne\hbox to\columnwidth{\hss\kern-6pc\box\@ne\hss}%
  \else 
    \setbox\@ne\vbox{\unvbox\@tempboxa\parskip\z@skip
        \noindent\unhbox\@ne\advance\hsize-6pc\par}%
\fi
  \ifnum\@tempcnta<64 
    \addvspace\abovecaptionskip
    \moveright 3pc\box\@ne
  \else 
    \moveright 3pc\box\@ne
    \nobreak
    \vskip\belowcaptionskip
  \fi
\relax
}
\makeatother
\def\writefig#1 #2 #3 {\rlap{\kern #1 truecm
\raise #2 truecm \hbox{#3}}}

\title[Universality for subcritical Complex Gaussian Multiplicative Chaos]{A universality result for subcritical Complex Gaussian Multiplicative Chaos}

\author{Hubert Lacoin}
\address{
  IMPA, Institudo de Matem\'atica Pura e Aplicada, Estrada Dona Castorina 110
Rio de Janeiro, CEP-22460-320, Brasil. 
}

\begin{document}

 \begin{abstract}
  In the present paper, we show that (under some minor technical assumption) Complex Gaussian Multiplicative Chaos defined as the complex exponential of a $\log$-correlated Gaussian field can be obtained by taking the limit of the exponential of the field convoluted with a smoothing kernel. We consider two types of chaos: $e^{\gamma X}$ for a log correlated field $X$ and $\gamma=\alpha+i\beta$, $\alpha, \beta\in \bbR$ and  $e^{\alpha X+i\beta Y}$ for $X$ and $Y$ two independent fields with $\alpha, \beta\in \bbR$. Our result is valid in the range 
   $$ \mathcal P_{\mathrm{sub}}:=\{ \alpha^2+\beta^2<d \} \cup \{ |\alpha|\in (\sqrt{d/2},\sqrt{2d} ) \text{ and } |\beta|< \sqrt{2d}-|\alpha| \},$$
   which, up to boundary, is conjectured to be  optimal.
   \\[10pt]
 \end{abstract}

\maketitle

 \section{Introduction}

 \subsection{Real Gaussian Multiplicative Chaos and the question of universality}
 
The theory of Gaussian multiplicative chaos (GMC) developped was developped  by Kahane \cite{zbMATH03960673} with the objective of giving a rigourous meaning to random measures of the type 
\begin{equation}\label{GMC}
 e^{\gamma X(x)-\frac{\gamma^2}{2}\bbE[(X(x))^2]} \nu(\dd x)
\end{equation}
where $X$ is a $\log$-correlated Gaussian field, that is,
a Gaussian field with a covariance function of the form
\begin{equation}\label{ladefdeL}
K(x,y)= \log \frac{1}{|x-y|}+L(x,y)
\end{equation}
where $L$ is continuous function
and $\nu$ is a finite measure, both defined on a bounded measurable set $D\subset\bbR^d$, and $\gamma$ is a positive real number. 
For the sake of simplicity, we assume in our discussion that $\nu$ is absolutely continuous  with respect to Lebesgue and with bounded density ($\nu(\dd x)= \varrho(x)\dd x$ where $\varrho$ is a positive bounded function and $\dd x$ denotes Lebesgue measure).
Motivations to define a random distribution corresponding to \eqref{GMC} are plenty and come from various fields such as fluid mechanics (study of turbulence), quantitative finance and mathematical physics (Conformal Field Theory). We refer to \cite{RVreview} for a detailed account of applications.

\medskip

Let us quickly expose the reasons why giving a meaning to \eqref{GMC} poses a mathematical challenge.
As the kernel $K$ diverges on the diagonal, the field $X$ can be defined only as a random distribution (see Section \ref{leledef} below): the quantity $X(x)$ is not well defined, and one can only make sense of $X$ integrated along suitable test functions.
To give a meaning to \eqref{GMC}, a possibility (and this is the original idea of Kahane's construction in \cite{zbMATH03960673}) is to consider a sequence $(X_n(x))_{x\in D}$ of functional approximations converging to $X$ and to consider the limit  
\begin{equation}\label{GMClimit}
\lim_{n \to \infty} e^{\gamma X_n(x)-\frac{\gamma^2}{2}\bbE[(X_n(x))^2]} \nu(\dd x),
\end{equation}
as the definition of GMC.

\medskip

In \cite{zbMATH03960673}
this approximation approach is sucessfully applied with the additional assumption that $K$ can be written in the form 
$K(x,y)=\sum_{k=1}^{\infty} Q_k(x,y)$ where $Q_{k}$ is a sequence of bounded  positive definite function satisfying $Q_k(x,y)\ge 0$ for every $x,y$.
This assumption allows in particular to approximate $X$ by a martingale sequence, by defining  $X_n=\sum_{k=1}^n Y_k$ where $Y_k$ is a sequence of independent fields, with respective covariance kernels  $Q_k(x,y)$. 
Under this assumption, it is shown in \cite{zbMATH03960673} that the limit  \eqref{GMClimit} exists for all $\gamma \in \bbR $, 
is nontrivial when $\gamma\in (-\sqrt{2d},\sqrt{2d})$ (this range of parameter has been referred to as the subcritical phase of the GMC) and is equal to $0$ when $|\gamma| \ge \sqrt{2d}$.
The result of Kahane yields a couple of natural questions:
\begin{itemize}
 \item [(A)] Is the limit obtained a function of $X$ or does it depend on the extra information which is present in the sequence $(X_n)_{n\ge 1}$? 
 \item [(B)] Would one obtain the same limit for some other kind of approximation of $X$ (e.g.\ considering convolution of $X$ by a smooth kernel)?
\end{itemize}
A positive answer to both questions is necessary to establish without a doubt that the construction in \cite{zbMATH03960673} as the natural definition of \eqref{GMC}.

\medskip

Let us focus on $(B)$ which is the question of universality and has been the object of studies through several decades (an extensive account on this is given in \cite{NatEle}).
A statement concerning universality \textit{in law} was proved in  \cite{robertvargas}. More precisely, it was shown that if one approximates $X$ with convolution by a smooth kernel, then the sequence \eqref{GMC} converges in law and that the law of the limiting object is independent of the convolution kernel used in the proceedure.

\medskip

More recent works \cite{NatEle, Shamov}  (see also \cite{junnila2015uniqueness}) gave a full answer to the universality question. In \cite{Shamov}, an axiomatic definition of Gaussian Multiplicative Chaos which allows to uniquely define \eqref{GMC} without the need of an approximation is given (in a setup which is much more general than the one considered here), and it is furthermore shown that for any reasonable notion of approximating sequence $(X_n)$, the sequence random measures in Equation \eqref{GMClimit} converges in probability to the object given by this axiomatic definition.
In \cite{NatEle}, it is established via elementary computations that every convolution approximation of the field yields the same limit in probability and that this limit is identical to the one obtained with the martingale approximation by Kahane.

\medskip

Note that this positive answer to $(B)$ also entails that the Gaussian multiplicative chaos is indeed only a function of $X$, thus providing an answer to $(A)$.

\subsection{Complex Gaussian Multiplicative Chaos}

More recently, Gaussian Multiplicative Chaos has been considered in a complex setup, the idea being to give a rigourous meaning to  $e^{\gamma X(x)-\frac{\gamma^2}{2}\bbE[(X(x))^2]} \nu(\dd x)$  for complex values of $\gamma$ \cite{KupiainenActa,BJM10, junnila2019} (see also \cite{BJM102,DES93,HK15,HK18} where hierarchical versions of the model are considered).
A variant of this problem \cite{LRV15} is to consider two independent $\log$-correlated Gaussian fields $X$ and $Y$ and consider the measure
\begin{equation}\label{indrealindim}
e^{\alpha X(x)+i\beta Y(x)-\frac{\gamma^2}{2}\bbE[(X(x))^2]+\frac{\beta^2}{2}\bbE[(Y(x))^2]} \nu(\dd x).
\end{equation}
Complex Gaussian Multiplicative Chaos found applications in random geometry \cite{MS16}, in the study of log-gases \cite{LRV19}. It also has connections with the Ising model \cite{junnila2018}, the Riemann Zeta function and random matrices \cite{saksman2016}.
We refer to the references mentionned above for further details and motivation.

\medskip

The main objective of this work is to establish a result similar to the one in \cite{NatEle} for complex GMC. In the case of complex $\gamma$ it has been shown in \cite{junnila2019} under some regularity assumption for $L$ in \eqref{ladefdeL} (more details are given below) that the real GMC admits an analytic continuation in an open domain which includes the real segment $(-\sqrt{2d}, \sqrt{2d})$. The domain is explicit (given by \eqref{subcrit} see also Figure \ref{lafigure}) and is optimal, in the sense that there are very strong heuristic evidences that convergence to a non trival limit cannot hold outside of the closure of this open set.
What we establish in the present paper is that under the same assumption, the approximation obtained by convoluting the field with a smooth kernel converges to this universal object.

\medskip

Concerning the case of independent real and imaginary part 
\eqref{indrealindim}, the existence of the limit has been proved for some martingale approximation under some restriction on the kernel $K$ (existence of an integral decomposition, see \cite{LRV15}). In the present work, we prove convergence of the approximation by convolution with no additional assumption on $K$ besides the fact that it is log-correlated. 

\medskip

Before introducing our results in more details, we provide a short and comprehensive technical introduction to GMC in the real and complex setup.

\medskip

\section{Setup and results}

\subsection{Log-correlated fields and their regular convolutions} \label{leledef}

Given an open set $\cD\subset \bbR^d$.
Consider $K$ a positive definite kernel defined on $\cD^2$ of the form
\begin{equation}\label{lexpress}
K(x,y)= \log \frac{1}{|x-y|}+L(x,y)
\end{equation}
where $L$ is continuous function on $\cD$. By positive definite, we mean that 
\begin{equation}\label{treks}
 \int_{\cD^2} K(x,y) f(x)f(y)\dd x \dd y\ge 0
\end{equation}
 for every  continuous $f$ with compact support. 
Using the same formalism as in \cite{NatEle}, we define the field $X$ with covariance function $K$ as a random process indexed by a set of signed measure.  
We define $\cM^+_K$ to be the set of positive borel measures on $\cD$ such that 
\begin{equation}
 \int_{\cD^2}  |K(x,y)|\mu(\dd x) \mu(\dd y) <\infty
\end{equation}
and let $\cM_K$ be the space of signed measure spanned by $\cM^+_K$
\begin{equation}
 \cM_K:=\left\{ \mu_+- \mu_- \ : \ \mu_+, \mu_- \in \cM^+_K \right\}.
\end{equation}
We define $\hat K$ as the following quadratic form on  $\cM_K$  
\begin{equation}\label{hatK}
 \hat K(\mu,\mu')=
\int_{\cD^2}  K(x,y)\mu(\dd x) \mu'(\dd y).
\end{equation}
The assumption \eqref{treks} ensures that  $\hat K$ is positive definite, in the sense that for any 
finite collection of measures $(\mu_i)^k_{i=1}$ in $\cM_K$,  $\hat K(\mu_i,\mu_j)_{1\le i,j\le k}$ is a positive definite matrix.
Finally let $X= (\langle X, \mu \rangle)_{\mu\in \cM_K}$ be the centered Gaussian process indexed by $\cM_K$ with covariance function given by $\hat K$.
Note that from \eqref{lexpress}, $\cM_K$ contains all compactly supported continuous functions. With some abuse of notation we identify the measure $m(x)\dd x$ with function $m(x)$ and write 
\begin{equation}
 \int_{\cD} X(x) m(x)\dd x:= \langle X, m \rangle
\end{equation}
We want to consider now an approximation of $X$ obtained by convolution with a smooth kernel.
Consider $\theta$ a non-negative $C^{\infty}$ function whose compact support is included in the Euclidean ball of radius one, and such that $\int_{B(0,1)} \theta (x) \dd x=1.$
We define for $\gep\in (0,1]$,
$\theta_{\gep}:=\frac{1}{\gep^d} \theta( \gep^{-1}\cdot),$
{\red and 
\begin{equation}\label{cdgep}
 \cD_{\gep}:= \{ x\in \cD \ : \ \min_{y \in \bbR^d \setminus\cD} |y-x|> 2\gep \}
\end{equation}
(or $\cD_{\gep}=\bbR^d$ is $\cD=\bbR^d$)}.
We introduce the convoluted field $X_{\gep}$ indexed by $\cD_{\gep}$ by setting 
\begin{equation}
 X_{\gep}(x):= \int_{\cD} X(y)\theta_{\gep}(x-y)\dd y.
\end{equation}
With this definition on can check that $X_{\gep}(x)$ is a centered Gaussian field indexed by ${\red \bbD:= \bigcup_{\gep\in(0,1]} \{\gep\}\times\cD_{\gep}}$ with covariance function 
\begin{equation}\label{labig}
K_{\gep,\gep'}(x,y):=\bbE[X_{\gep}(x)X_{\gep'}(y)]=
\int_{(\bbR^d)^2} \theta_{\gep}(x-z_1) \theta_{\gep'}(y-z_2)
K(z_1,z_2)\dd z_1 \dd z_2.
\end{equation}
We simply write $K_{\gep}$ when $\gep=\gep'$, and $K_{\gep}(x)$ when $x=y$.
Finally 
$K_{\gep,\gep'}(x,y)$ is sufficiently regular (that is, both H\"older continuous in $x$ and $\gep$) to apply Kolmogorov criterion (see e.g.\  \cite[Theorem 2.9]{legallSto}).
Thus, in particular, there exists a version of the field which is jointly continuous in $\gep$ and $x$. In what follows we will always be considering this continuous version of the field.

\subsection{Gaussian multiplicative chaos in the complex case}

 Given $K$ satisfying \eqref{lexpress}, we consider $X$ a Gaussian field with covariance $K$, 
we consider
  $(X_{\gep}(x))_{(\gep,x)\in \mathbb D}$ a continuous version in $\gep$ and $x$ of the mollified field, and $\nu$ a locally finite Borel measure on $\cD$. 
 We define
the $\gep$-mollified Gaussian Multiplicative chaos associated with $X$ , $\nu$ and $\gamma=\alpha+i\beta \in \bbC$ by (recall that with our notation  $K_{\gep}(x)=\bbE\left[  (X_{\gep}(x))^2\right]$) as follows: 
For any function $f\in C_c(\cD)$  (continuous on $\cD$ with compact support)
we set 
\begin{equation}\label{lesssimple}
 M^{(\gamma)}_{\gep}(f)= \int_{\cD_{\gep}} e^{\gamma X_{\gep}(x)-\frac{\gamma^2}{2}K_{\gep}(x)}f(x)\nu( \dd x).
\end{equation}
{\red The restriction to $\cD_{\gep}$ not only ensures that $X_{\gep}(x)$ is well defined, but also avoids boundary effects to ensure integrability: $X_{\gep}(x)$ and  $K_{\gep}(x)$ are uniformly bounded on $\cD_{\gep}$. When the support of $f$ is included in $\cD_{\gep}$ we will, with a small abuse of notation, write  \eqref{lesssimple}  as an integral over $\cD$.}
A variant of the model with independent real and imaginary parts of the field in the exponential can also be considered.
Given $\alpha$ and $\beta$ two real numbers, $X$ and $Y$ two independent fields with covariance $K$ we set for  $f\in C_c(\cD)$
\begin{equation}\label{simplecase}
 M^{(\alpha,\beta)}_{\gep}(f):= \int_{\cD_{\gep}} e^{\alpha X_\gep(x)+i\beta Y_\gep(x)+\frac{\beta^2-\alpha^2}{2} K_{\gep}(x)} f(x) \nu(\dd x).
\end{equation}
We are interested in the limit when $\gep$ tends to zero of the quantities defined above.
More specifically we want to show that, within some range for the parameters $\alpha$ and $\beta$, $M_{\gep}$ converges to a non-trivial limit which does not depend on the convolution kernel $\theta$.
As mentionned above in the introduction,  such a result has been proved in the real case (when $\beta=0$ since $\gamma=\alpha$ simply write  $M^{(\alpha)}_{\gep}$). Let us mention this result as it is  found in \cite{NatEle}.
For the remainder of the paper we will assume that 
$\nu(\dd x)=\varrho(x)\dd x$ where $\varrho$ is a bounded measurable function on $\cD$ (note that  \cite{NatEle} allows for some flexibility on the choice of the measure $\nu(\dd x)$ but we have chosen to keep the setup as simple as possible here).

\begin{theorema}\label{darealcase}
Let $\alpha \in (-\sqrt{2d},\sqrt{2d})$ be a real number. 
Then
for $M^{(\alpha)}_{\gep}(f)$ defined as in \eqref{lesssimple}  we have
the following convergence in probability and in $\bbL_1$ for every $f\in C_c(\cD)$

\begin{equation}
   \lim_{\gep\to 0} M^{(\alpha)}_{\gep}(f)=M^{(\alpha)}_0(f).
\end{equation}
where $M^{(\alpha)}_0(f)>0$ almost surely if $f$ is non-negative and non uniformly zero.
 Furthermore the limit does not depend on the choice of  the smoothing kernel $\theta$.
\end{theorema}
Note that the range of parameter $\alpha$ considered above is optimal since it is known that when $|\alpha|\ge \sqrt{2d}$ we have $
    \lim_{\gep\to 0} M^{(\alpha)}_{\gep}=0,$
in probability (see e.g.\ \cite[Proposition 3.1]{robertvargas}).
In the complex setup, we are focusing on the so-called \textit{subcritical phase} which corresponds to the following range for the parameter $\alpha$ and $\beta$
\begin{equation}\label{subcrit}
 \mathcal P_{\mathrm{sub}}:=\{ \alpha^2+\beta^2<d \} \cup \{ |\alpha|\in (\sqrt{d/2},\sqrt{2d} ) \text{ and } |\beta|< \sqrt{2d}-|\alpha| \}.
\end{equation}
In words,  $\mathcal P_{\mathrm{sub}}$ is the convex envelope of the union of the ball of radius $\sqrt{d}$ and the segment $(-\sqrt{2d},\sqrt{2d})\times \{0\}$ (see Figure \ref{lafigure}).
Our aim is to extend Theorem \ref{darealcase} to the complex setup, in the subcritical case.

 \begin{figure}[ht]
\begin{center}
\leavevmode
\epsfxsize = 10 cm
\psfrag{d}{$\sqrt{d}$}
\psfrag{2d}{$\sqrt{2d}$}
\psfrag{beta}{$\beta$}
\psfrag{alpha}{$\alpha$}
\epsfbox{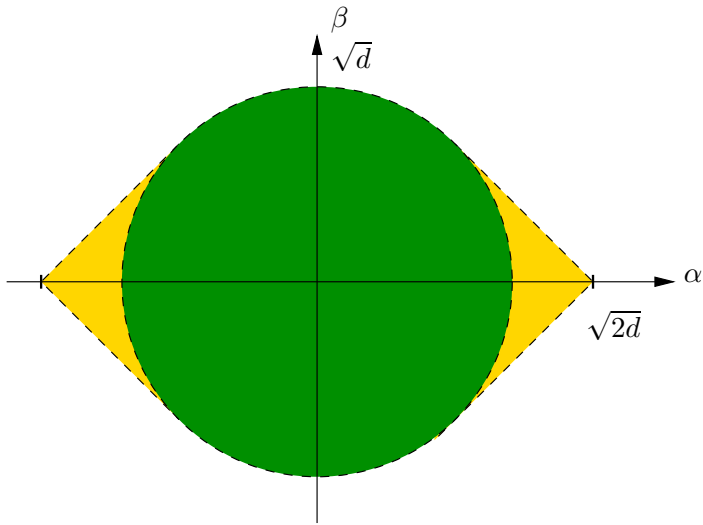}
\end{center}
\caption{\label{lafigure} 
The domain $\cP_{\mathrm{sub}}$. The green (dark) region $\alpha^2+\beta^2<d$ corresponds to the $\bbL_2$ region for which the proof of convergence is relatively straighforward (see Section \ref{l2sec}).
The yellow (lighter) region corresponds to the zone where a more advanced proof is required (and presented in Sections \ref{secdecoco} and \ref{complexgamma})}
\end{figure}

\medskip

\noindent Let us mention that when $(\alpha,\beta)\in \mathcal P_{\mathrm{sub}}$ (and under some additional assumption on the kernel $K$) the existence of a random distribution corresponding the formal expressions
\begin{equation}
  e^{\alpha  X(x)+ i\beta Y(x)+\frac{\beta^2-\alpha^2}{2}\bbE[(X(x))^2]} \nu(\dd x) \quad  \text{ and } \quad  e^{\gamma X(x)-\frac{\gamma^2}{2}\bbE[(X(x))^2]} \nu(\dd x)
\end{equation}
was established in
in \cite{LRV15} and \cite{junnila2019} respectively. In both cases, the construction relies on a martingale approximation of the field $X$ similar to Kahane's construction. What we establish in the present paper is that any convolution approximation of the field yields the same object in the limit.

\subsection{Results}

\subsubsection*{Convergence of $M^{(\alpha,\beta)}_{\gep}(f)$}

\begin{theorem}\label{decoco}
If $(\alpha,\beta)\in \mathcal P_{\mathrm{sub}}$, $f\in C_c(\cD)$ and
 $M^{(\alpha,\beta)}_{\gep}(f)$ is defined as in \eqref{simplecase}, then the following limit exists in probability and in $\bbL_1$ 
 \begin{equation}
  \lim_{\gep\to 0} M^{(\alpha,\beta)}_{\gep}(f)=M^{(\alpha,\beta)}_0(f).
 \end{equation}
Furthermore the limit does not depend on the choice of the smoothing kernel $\theta$.
\end{theorem}

\noindent Note that the convergence in $\bbL_1$ implies that 
\begin{equation}
 \bbE \left[ M^{(\alpha,\beta)}_{0}(f)\right]= \lim_{\gep\to 0} \bbE\left[ M^{(\alpha,\beta)}_{\gep}(f) \right]=\int_{\cD} f(x) \varrho(x) \dd x,
\end{equation}
which indicates that the limit is non trivial.

\subsubsection*{Convergence of $M^{(\gamma)}_{\gep}(f)$}

In the case of a single complex parameter $\gamma$, we require and extra regularity assumption on $K$ (which comes from \cite{junnila2019}).
More specifically we are going to assume that $K$ can be written in the form \eqref{lexpress} 
where the function $L$ belong to the local Sobolev space 
$H^{s}_{\mathrm{loc}}(\cD\times \cD)$ for some $s>d$. 
For $k\ge 1$, the Sobolev space  $H^{s}(\bbR^k)$  is the Hilbert space associated with the norm
\begin{equation}
\|\varphi \|_{H^s(\bbR^k)}:=  \left(\int_{\bbR^k} (1+|\xi|^2)^s |\hat \varphi(\xi)|^2 \dd \xi < \infty \right)^{1/2},
\end{equation}
where $\hat \varphi$ denotes the Fourier transform of $\varphi$ defined for smooth functions by
\begin{equation}
 \hat \varphi(\xi):= \int_{\bbR^k} e^{i \xi. x} \varphi(x) \dd x.
\end{equation}
For an open set $U\subset \bbR^k$,  $H^{s}_{\mathrm{loc}}(U)$ denotes the set function which belongs to  $H^{s}(\bbR^k)$ after multiplication by an arbitrary smooth function with compact support
\begin{equation}
H^{s}_{\mathrm{loc}}(U):= \left\{  \varphi : U \to \bbR  \ | \ \forall \rho\in C^{\infty}_c(U), \  \rho\varphi\in
 H^{s}(\bbR^k)  \right\},
\end{equation}
where with some abuse of notation, $\rho\varphi$ is identified with its extension by zero on $\bbR^k$.
\begin{theorem}\label{coco}
Assuming that $K$ is of the form \eqref{lexpress} for a function 
$L\in H^{s}_{\mathrm{loc}}(\cD\times \cD)$, $s>d$.
If $(\alpha,\beta)\in \mathcal P_{\mathrm{sub}}$, $\gamma=\alpha+i\beta$, $f\in C_c(\cD)$ and 
$M^{(\gamma)}_{\gep}(f)$ is defined as in \eqref{lesssimple}, the following limit exists in probability and in $\bbL_1$
 
 \begin{equation}
  \lim_{\gep\to 0} M^{(\gamma)}_{\gep}(f)=M^{(\gamma)}_0(f).
 \end{equation}
 Furthermore the limit does not depend on the choice of the smoothing kernel $\theta$.
\end{theorem}

\subsubsection*{Convergence as distributions}

{\red Theorem \ref{coco} and \ref{decoco} concerns the convergence of the chaos integrated over a  function $f$ considered as a random variable. It is possible to go further and prove that 
  $M^{(\alpha,\beta)}_{\gep}$ and $M^{(\gamma)}_{\gep}$
converge towards a limiting random  distribution.
The convergence holds in a Sobolev space of negative index, and in particular this means that a priori not every continuous function can be integrated against the limit $M^{(\gamma)}_0$.

\begin{theorem}\label{asadistrib}
If the assumptions of Theorem  \ref{decoco} or \ref{coco}  are satisfied then  for any $u>d/2$,
 $M^{(\alpha,\beta)}_{\gep}$ and $M^{(\gamma)}_{\gep}$
 converge in the local Sobolev space $H^{-u}_{\mathrm{loc}}(\cD)$ towards respective limiting distribution $M^{(\alpha,\beta)}_{0}$ and  $M^{(\gamma)}_{0}$.
 These convergences holds in probability.

\end{theorem}


}

\begin{rem}\label{wopopop}
 While in \cite{junnila2019}, the complex GMC is not obtained using smoothing kernels, it is worthwhile mentionning that the limit defined above coincides with the complex GMC constructed in  \cite{junnila2019}. This follows from the uniqueness of the limit on the real line and analyticity in $\gamma$ (we refer to \cite{junnila2019} for details on how to prove analyticity). 
 In the same manner, the limit presented in Theorem \ref{decoco} coincides with the one defined \cite[Theorem 3.1]{LRV15}. Some details about this last point are given in Section \ref{nodepend}.
\end{rem}

\begin{rem}
{\red While most examples of log-correlated fields considered in the litterature satisfy it, the assumption $L\in H^{s}_{\mathrm{loc}}(\cD)$ for some $s>d$ is a genuine restriction. There exist positive kernels of the form \eqref{ladefdeL}, for which $L\notin H^{s}_{\mathrm{loc}}(\cD)$
for any $s>d/2$. This is the case of the kernel $K$ defined by 
$$K(x,y):=\int^{\infty}_0 (1+t^{-2}) \kappa(e^t|x-y|)\dd t.$$
with $\kappa$ as in \eqref{defkappa}. 
}
\end{rem}

\subsection{Possible extensions of the result, open problems and related work}

We have chosen to keep the setup as simple as can be for the ease of the exposition but let us mention here some small extension that can be obtained with only minor modifications in the proof.

\subsubsection*{Correlated real and imaginary part}
In \cite{ HK15, HK18, KK14} the case of multiplicative cascades with correlated real and imaginary part is also considered.
In our context this corresponds to considering 
$X$ and $Y$ with covariance $K$ and such that the covariance between $X$ and $Y$ is given by $p K$ for some fixed $p\in (-1,1)$. That is
$$\bbE \left[\langle X, \mu \rangle  \langle Y, \mu '\rangle \right]:= p \int K(x,y) \dd \mu(\dd x) \dd \mu'(\dd y).  $$
In that case, the tecniques we develop for the proof of Theorem \ref{coco} (in Section \ref{complexgamma}) fully adapts (without any need for change) under the same assumption for $K$ (that is $L \in H^s_{\mathrm{loc}}(\cD)$ for some $s>d$).
{\red
In fact we do not require $X$ and $Y$ to have the same marginal law.
The most general case that can be treated without substancial modification to the proof is 
covariance of the form
\begin{equation}
\begin{split}
  \bbE[X(x)X(y)]&=K_1(x,y)= \log \frac{1}{|x-y|} +L_1(x,y),\\
  \bbE[Y(x)Y(y)]&=K_2(x,y)= \log \frac{1}{|x-y|} +L_2(x,y),\\
  \bbE[X(x)Y(y)]&=K_3(x,y)=p\log \frac{1}{|x-y|} +L_3(x,y).
\end{split}
  \end{equation}
 with $L_1$, $L_2$ and $L_3$ in $H^{s}_{\mathrm{loc}}(\cD)$, $p\in (-1,1)$ and the adequate positive definiteness assumption that is for $f$ and $g$ in $C_c(\cD)$
 \begin{equation}\label{posdef}
  \int_{\cD^2} \left[K_1(x,y)f(x) f(y) + K_2(x,y)g(x) g(y)+K_3(x,y)f(x) g(y)\right] \dd x \dd y\ge 0.
 \end{equation}
The important point to check in that case is that we can have a martingale 
 decompositions  like the one introduced in Section \ref{lasecsec} for  $X$ and $Y$ 
 \textit{with a common filtration}. This can be deduced from Proposition \ref{dapopo}.}.
\subsubsection*{More general reference measures $\nu$}

We restricted our study to measures which are absolutely continuous with respect to Lebesgue. This assumption can be relaxed, and we can adapt our proof to a setup as general as the one considered in the real case \cite{NatEle}. More precisely, considering  $d'\in (0,d]$ and assuming that the measure $\nu$ satisfies 
$\int_{D\times D} \frac{1}{|x-y|^{d'}}\nu(\dd x )\nu(\dd y)<\infty,$
then  we have  convergence of  $M^{(\alpha,\beta)}_{\gep}$ and $M^{(\gamma)}_{\gep}$ as soon as 
$$  \alpha^2+\beta^2<d'  \quad \text{ or } \quad  |\alpha|\in (\sqrt{d'/2},\sqrt{2d'} ) \text{ and } |\beta|< \sqrt{2d'}-|\alpha|.$$

\medskip

\subsubsection*{Regularity of $M^{(\gamma)}$ as a distribution}

{\red 
In \cite{junnila2019regularity}, the authors investigated the regularity of $M^{(\gamma)}$  as a distribution, and proved that the distribution $M^{(\gamma)}$ is in fact more regular than the Sobolev regularity given Theorem \ref{asadistrib} - that is, $M^{(\alpha,\beta)}_{0}, M^{(\gamma)}_{0}$ both belong to  $H^{-u}_{\mathrm{loc}}$ for $u>d/2$. 
 The results \cite[Theorems 3 and 4]{junnila2019regularity} establishes a finer Besov regularity, with parameters which   depend on the value of $\gamma$.
It is natural to expect that $M^{(\gamma)}_{\gep}$ should also converge in these Besov functional space but this is out of the focus of the present paper.}

{\red
\subsubsection*{Remainder of the phase diagram}
The domain $\mathcal P_{\mathrm{sub}}$ corresponds to the subcritical regime of the Complex Gaussian Multiplicative Chaos (also called phase I in \cite{LRV15}). This is one of three phases which appear in the phase diagram of the model (this diagram appears for several related models e.g. \cite{DES93, HK18, KK14, LRV15}).
When $\gamma$ belongs to another phase, it is conjectured that $M^{(\gamma)}_\gep$ requires to be renormalized by a power of $\gep$ in order to converge to a non-trivial limit. 
Furthermore, in this case, the convergence only holds in the distributionnal sense and there is no almost sure convergence.
In \cite{Lac20}, the case of the so-called third phase
\begin{equation}
\cP_{\mathrm{III}}:= \{ \alpha+i \beta \ : \alpha,\beta \in \bbR, \ |\alpha|<\sqrt{d/2},\  \alpha^2+\beta^2> d  \}
\end{equation}
is treated and it is shown in that case that $\gep^{\frac{|\gamma|^2-d}{2}}M^{(\gamma)}_\gep$ converges in law (but not in probability)
to a complex Gaussian white noise with a random intensity, which is given by the real multiplicative chaos $M^{(2\alpha)}_0$.
The remaining part of the phase diagramm corrresponds to $\cP_{\mathrm{II}}$ which is also refered to as the  \textit{glassy phase}
$$\cP_{\mathrm{II}}:= \left\{\alpha+i \beta \ : |\alpha|+|\beta|>\sqrt{2d} \ ; \  |\alpha|> \sqrt{d/2} \right\}.$$
In this regime, it is conjectured that $(\log 1/\gep)^{\frac{3\alpha}{2}}\gep^{\sqrt{2d}\alpha-d} M^{(\gamma)}_\gep$ converges (also only in law) to a non-trivial limit. The limit should be purely atomic (i.e.\ be a weighted sum of Dirac masse).
A result has been proved in this direction when $\gamma\in \bbR$ (that is $\beta=0$, $|\alpha|>\sqrt{2d}$) in \cite{madaule16} although not for the convolution approximation of the field.
This phase  has also been investigated in \cite{HK15,madaulebis} for the related Branching Brownian Motion energy model.  }

\subsubsection*{Convergence on a part of the boundary of $\mathcal P_{\mathrm{sub}}$}
As mentionned in the introduction, the range of parameter $\cP_{\mathrm{sub}}$ is almost optimal for the convergence problem. Indeed,  the phase diagramm presented in \cite{LRV15} (which was discovered earlier in \cite{DES93} for the hierachical version of the model, see also \cite{KK14, HK18}) indicates that the limit of  $M^{(\alpha,\beta)}_{\gep}$ (and by analogy also $M^{(\gamma)}$) does not exist or  is  degenerate on the complement of the closure of $\cP_{\mathrm{sub}}$.
The boundary case is more delicate but \cite{LRV15} indicates that $M^{(\alpha,\beta)}_{\gep}$ and $M^{(\gamma)}_{\gep}$ should converge to a non-trivial limit only when
when $|\beta|=\sqrt{2d}-|\alpha|$, $|\alpha|\in (\sqrt{d/2},\sqrt{d})$, the other boundary cases require an other scaling and have a limit of a different nature.
Proving this rigourously and in full generality remains a challenging task.

\subsubsection*{Non Gaussian Chaos}
Multiplicative chaos has been studied beyond the Gaussian setup (see e.g. \cite{BJM10}). In  \cite{junnila2016m} the author investigated the complex exponential of Nongaussian fields with a $\log$ correlated structure, such as Fourrier series with Nongaussian random coefficients, and established the convergence of a martingale approximation in a complex domain which contains the segment $(-\sqrt{2d},\sqrt{2d})$. The fields can also be approximated by convolution and it remains an open question whether the convolution approximation converges and yields the same limit. The method presented here cannot operate outside of Gaussian setups since it heavily relies on the Cameron-Martin formula. 
Another interesting question is whether these Nongaussian chaos remain convergent in the full domain $\cP_{\mathrm{sub}}$.

\medskip

\subsection{Organization of the paper}

In the short  Section \ref{l2sec} we expose the argument which entails convergence in the case $\alpha^2+\beta^2<d$ (the so called $\bbL_2$ region). The argument is not new, but we include it since it is very short and yield some information about the proof strategy in the other cases.
In Section \ref{secdecoco}, we prove Theorem \ref{decoco} and in Section \ref{complexgamma} we prove Theorem \ref{coco}. The two proof are are partially inspired by the method used in \cite{NatEle}, though they present significant novelty. The proof of Theorem \ref{decoco} and that of Theorem \ref{coco}  share some common ideas, but the case of complex $\gamma$ requires some more advanced strategy. 
These sections are placed in increasing order of technical difficulty and should be read in that order. Finally In Section \ref{tightness}, we prove Theorem \ref{asadistrib} for the sequence $(M^{(\gamma)}_\gep)_{\gep>0}$ (the case $M^{(\alpha,\beta)}_\gep$ can be treated similarly).

\medskip

{\red
 {\bf Notational convention: } Throughout the paper, many inequalities are valid up to an additive or multiplicative constant. We use the generic letter $C$ for these constants and the value of $C$ is allowed to vary from one equation to another. The set $\cD$, the covariance kernel $K$ and the smoothing convolution kernel $\theta$  are considered fixed once and for all, and the constants denoted by $C$ are allowed to depend on these parameters.
 To simplify the notation, we assume that  $\nu$ is simply the Lebesgue measure, but the case $\nu(x)=\varrho(x)\dd x$ with $\varrho$ bounded  does not require any modification.  }

\section{The $\bbL_2$ convergence when $|\gamma|<\sqrt{d}$}
 \label{l2sec}

Let us display in this section the full proof of the convergence of  $M^{(\gamma)}_{\gep}(f)$, $f\in C_c(\cD)$ when $|\gamma|<\sqrt{d}$  (the same proof also applies to $M^{(\alpha,\beta)}_{\gep}(f)$ in the same range of parameters).
While this is not a new result (or proof), we have not seen it written up in details elsewhere in this context, and  it may provide to the reader some insight for the techniques used in the next sections.

\begin{proposition}\label{lecasl2}
 If $\gamma\in \bbC$ satisfies $|\gamma|<\sqrt{d}$, then the following limit exists in $\bbL_2$
 
 \begin{equation}
  \lim_{\gep\to 0} M^{(\gamma)}_{\gep}(f)=M^{(\gamma)}_0(f).
 \end{equation}
 Furthermore the limit does not depend on the choice of  the smoothing kernel $\theta$.

\end{proposition}
Our proof is going to rely on an estimate for the
correlation kernel $K_{\gep,\gep'}(x,y)$ (recall the definition \eqref{labig}.
The proof is standard and left to the reader (note that since $L$ is continuous and thus bounded on the considered set , it is sufficent to prove \eqref{stimatz}  for $K(x,y)=\log \frac{1}{|x-y|}$).

\begin{lemma}\label{anexo}
With the setup described above, given $\eta>0$ and $R>0$ there exists a positive constant $C_{\eta,R}>0$ such that for any $\gep,\gep' \in(0, \eta]$
and any $x,y \in  \cD_{\eta}\cup B(0,R)$
\begin{equation}\label{stimatz}
  \left| K_{\gep,\gep'}(x,y)- \log \frac{1}{|x-y|\vee \gep \vee \gep'} \right| \le  C_{\eta,R}
\end{equation}
and we have furthermore if $x,y\in \cD$ and  $x\ne y$
\begin{equation}\label{laconv}
\lim_{\gep,\gep'\to 0} K_{\gep,\gep'}(x,y)=K(x,y).
\end{equation}
\end{lemma}

\begin{proof}[Proof of Proposition \ref{lecasl2}]
Since $f$  and $\gamma$ is fixed, with a small abuse of notation, we  simply write $M_\gep$ for $M^{(\gamma)}_{\gep}(f)$
It is sufficient to prove that the sequence is Cauchy in $\bbL_2$. We have 
\begin{equation}\label{decomp}
\bbE\left[ |M_\gep-M_{\gep'}|^2\right] \!\!\!
= \bbE \left[ |M_\gep|^2 \right]+ \bbE \left[ |M_{\gep'}|^2 \right]- \bbE \left[M_\gep \overline{M} _{\gep'}\right]-\bbE \left[\bar M_\gep {M} _{\gep'}\right].
 \end{equation}
Hence it is sufficient to show that $\bbE \left[M_\gep \overline{M} _{\gep'}\right]$ converges when $\gep$ and $\gep'$ both go to zero (this implies that the four terms in the r.h.s.\ of \eqref{decomp} cancel out in the limit).
Assuming that $\gep$ and $\gep'$ are sufficiently small so that the support of $f$ is included in  $\cD_{\gep\vee \gep'}$ (recall \eqref{cdgep}) we have 
\begin{multline}
 \bbE \left[M_\gep \overline{M} _{\gep'}\right]
 =\int_{\cD^2} \bbE\left[ e^{\gamma X_{\gep}(x)+\bar \gamma X_{\gep}(y)-
 \frac{\gamma^2 K_{\gep}(x)+\bar \gamma^2  K_{\gep'}(y)}{2}}\right]f(x)f(y) \dd x \dd y\\
 =\int_{\cD^2} e^{|\gamma|^2 K_{\gep,\gep'}(x,y)}f(x)f(y)\dd x \dd y.
\end{multline}
From Lemma \ref{anexo}, we have $ e^{|\gamma|^2 K_{\gep,\gep'}(x,y)}\le C|x-y|^{-|\gamma|^2}$ when $x$ and $y$ are in the support of $f$ and thus  we obtain by dominated convergence 
\begin{equation}
 \lim_{\gep,\gep'\to 0} \int_{\cD^2} e^{|\gamma|^2 K_{\gep,\gep'}(x,y)}f(x)f(y)\dd x \dd y= \int_{\cD^2} e^{|\gamma|^2 K(x,y)}f(x)f(y)\dd x \dd y.
\end{equation}
\end{proof}

\section{Proof of Theorem \ref{decoco}}\label{secdecoco}

\subsection{The strategy of proof}\label{secstratz}

In this section we prove Theorem \ref{decoco}. The proof builds on  the ideas developped in \cite{NatEle} to prove Theorem \ref{darealcase}, the main one being to consider a ``truncated'' version of $M^{(\alpha,\beta)}_{\gep}(f)$ by discarding the contribution of excessively high values of $X_{\gep}$.
However, there is a key difference here. In \cite{NatEle}, it is shown that the difference between the truncated partition function and the original one is small in $\bbL_1$. This is not possible to show this in the complex case and we have to make sure that our truncated partition function exactly coincides with the original one with a probability which tends to one when the truncation level goes to infinity.
Our result is proved by showing that:
\begin{itemize}
 \item [(A)] The truncated version of the partition function converges in $\bbL_2$,
 \item [(B)] With a large probability the truncated and non-truncated version of the partition function coincide.
\end{itemize}
Note that without loss of generality we can assume that $\alpha$ and $\beta$ are both non-negative.
Let us assume that $(\alpha,\beta)\in \cP_{\mathrm{sub}}$ with  $\alpha\in(\sqrt{d/2},\sqrt{2d})$ and $\beta>0$ (the other case can be treated with the $\bbL_2$ method as in Proposition \ref{lecasl2}). We fix  $\gl>0$ that satisfies 
\begin{equation}\label{surlambda}
 \sqrt{2d}<\gl<2\alpha \quad \text{ and } \quad  d +\frac{(2\alpha-\gl)^2}{2}> \alpha^2+\beta^2.
\end{equation}
The reader can check that the existence of such a $\gl$ follows from our assumptions. 
For $k\ge 1$ we define (with some minor abuse of notation) $X_k:=X_{\gep_k}$ where $\gep_k=e^{-k}$. {\red For any integer $q$ such that the support of $f$, satisfies $\Supp(f)\subset \cD_{\gep_q}$ (we let $q_0(f)$ denote the smallest such integer)}
we define for $q\ge q_0(f)$ the events, $A_{q,\gl}(x)$ for $x\in \cD_{\gep_k}$, ) and $\cA_{q,\gl}(f)$ as 

\begin{equation}\label{defaqgl}\begin{split}
                 A_{q,\gl}(x)&:= \left\{  \forall k\ge q, \quad  X_k(x) \le  k \gl  \right\},\\
                  \cA_{q,\gl}(f) := \bigcap_{ x\in \Supp(f)}  & A_{q,\gl}(x)= \left\{  \forall k\ge q, \quad  \sup_{x\in  \Supp(f)} X_k(x) \le  k \gl  \right\}.
                \end{split}
\end{equation}
Now we define $M^{(\alpha,\beta)}_{\gep,q}(f)$ (we will omit the dependence in $\alpha$ and $\beta$ most of the time to alleviate the notation) by 
\begin{equation}
 M^{(\alpha,\beta)}_{\gep,q }(f):= \int_{\cD_{\gep}} e^{\alpha X_\gep(x)+i\beta Y_\gep(x)+\frac{\beta^2-\alpha^2}{2} K_{\gep}(x)} \ind_{ A_{q,\gl}(x)}f(x)\dd x.
 \end{equation}
The convergence of   $M^{(\alpha,\beta)}_{\gep}$ is deduced from the two following statements. 
 \begin{proposition}\label{firstacase}
Given $f\in C_{c}(\cD)$, $q\ge q_0(f)$ the sequence  $(M^{(\alpha,\beta)}_{\gep,q }(f))_{\gep\in(0,1]}$ is Cauchy in $\bbL_2$. In particular the following limit exists
\begin{equation}
 \lim_{\gep\to 0}M^{(\alpha,\beta)}_{\gep,q }(f):= M^{(\alpha,\beta)}_{0,q }(f).
\end{equation}
Furthermore the limit does not depend on the choice of $\theta$.
  
 \end{proposition}

\begin{proposition}\label{lapropo}
If  $\gl>\sqrt{2d}$ then we have 
 \begin{equation}
  \lim_{q\to \infty} \bbP[\cA_{q,\gl}(f)]=1.
 \end{equation}

\end{proposition}
The proof of Proposition \ref{firstacase} is detailed in the next subsection. 
The asymptotic of the maximum of log-correlated Gaussian fields is a much studied topic and results that are  much more precise that Proposition \ref{lapropo} have been proved for related models (see for instance \cite{Mad15,BDZ16,BL16}).
We could not find a reference that matches the setup considered in the present paper, and for this reason, we include a proof in Appendix \ref{prooflapropo}.

\begin{proof}[Proof of Theorem \ref{decoco} from Proposition \ref{firstacase} and \ref{lapropo}]
 
Let $\mathfrak q$ be the smallest integer value of $q$ such that $\cA_{q,\gl}$ holds. Proposition \ref{lapropo} implies that $\mathfrak q$ is finite almost surely.
We have for every $\gep$, $M^{(\alpha,\beta)}_{\gep}(f)= M^{(\alpha,\beta)}_{\gep,\mathfrak q}(f),$
and thus as consequence of Proposition \ref{firstacase}, 
$M^{(\alpha,\beta)}_{\gep}(f)$ converges in probability towards 
$M^{(\alpha,\beta)}_{0,\mathfrak q}(f)=M^{(\alpha,\beta)}_0(f)$ ($\mathfrak q$ is a random variable but the convergence in probability can be obtained by decomposing on all its possible values since there are only countably many).
\end{proof}

\subsection{Proof of Proposition \ref{firstacase}}

We first prove the convergence result for a fixed $\theta$ and discuss the dependence in $\theta$ (which turns out to be direct consequence of the proof) in Section \ref{nodepend}.
For the same reason as in \eqref{decomp}, we only need to prove the convergence of $\bbE\left[ M^{(\alpha,\beta)}_{\gep,q }(f)\bar M^{(\alpha,\beta)}_{\gep',q }(f) \right]$ towards a finite limit. We omit the dependence in $f$, $\alpha$ and $\beta$ in the computation.
Let us assume that $\gep'\le \gep$ and that the support of $f$ is included in $\cD_{\gep}$ (recall \eqref{cdgep}).
 Averaging first we respect to $Y$, and setting  $A_q(x,y)=A_{q,\gl}(x)\cap A_{q,\gl}(y)$ we obtain 
\begin{equation}\begin{split}\label{ziut}
 \bbE&\left[  M_{\gep,q }\bar M_{\gep',q } \right]\\
 &= \int_{\cD^2} e^{\beta^2 K_{\gep,\gep'}(x,y)} \bbE\left[ e^{\alpha (X_{\gep}(x)+ X_{\gep'}(y))-
 \frac{\alpha^2}{2}(K_{\gep}(x)+K_{\gep'}(y))} \ind_{A_q(x,y)}\right]f(x)f(y)\dd x \dd y
 \\&= \int_{\cD^2} e^{(\alpha^2+\beta^2) K_{\gep,\gep'}(x,y)} \tilde \bbP_{\gep,\gep',x,y}(A_q(x,y))f(x)f(y)\dd x \dd y,
\end{split}
 \end{equation}
where  $\tilde \bbP_{\gep,\gep',x,y}$ is defined by its density with respect to $\bbP$ which is equal to 
\begin{equation}\label{cameron}
 \frac{\dd \tilde \bbP_{\gep,\gep',x,y}}{\dd \bbP}=
 e^{\alpha X_{\gep}(x)+\alpha X_{\gep'}(y)-
 \frac{\alpha^2}{2}[K_{\gep}(x)+K'_{\gep}(y)+2K_{\gep,\gep'}(x,y)]}.
\end{equation}
We conclude from \eqref{ziut} using dominated convergence theorem and the following estimate for $\tilde \bbP_{\gep,\gep',x,y}(A_q(x,y))$.

\begin{lemma}\label{ladomine}
The following domination and convergence results hold.
\begin{itemize}
  \item  [(A)]There exists a constant $C_q>0$ such that for every $x,y\in D$,  if $\gep'\le \gep$ we have
  \begin{equation}\label{ladomi}
   \tilde \bbP_{\gep,\gep',x,y}(A_q(x,y))\le C_q (|x-y|\vee \gep)^{ \frac{(2\alpha-\gl)^2}{2} } 
  \end{equation}

  \item [(B)]  We have $ \lim_{\gep,\gep'\to 0} \tilde \bbP_{\gep,\gep',x,y}(A_q(x,y)) =\bbP[ \bar A_{q}(x,y)]$,
  where 
  \begin{equation*}
    \bar A_{q}(x,y):= \bigcap_{k\ge q}\left\{ X_k(x) \le  k \gl -\alpha H_k(x,y)\ ; \  X_k(y) \le  k \gl -\alpha H_k(y,x) \right\},
  \end{equation*}
and $H_k(x,y):= K_{\gep_k,0}(x,x)+ K_{\gep_k,0}(y,x)$ ($\gep_k=e^{-k}$ and $K_{\gep,0}$ is defined by \eqref{labig}).
 \end{itemize}
\end{lemma}
\noindent By Lemma \ref{anexo} and Lemma \ref{ladomine}, the integrand in the r.h.s. of \eqref{ziut} satisfies
$$ e^{(\alpha^2+\beta^2) K_{\gep,\gep'}(x,y)} \tilde \bbP_{\gep,\gep',x,y}(A_q(x,y))\le  C'_q |x-y|^{ \frac{(2\alpha-\gl)^2}{2}-(\alpha^2+\beta^2)},$$
for some constant positive constant $C'_q$,
and thus is integrable due to the assumption \eqref{surlambda}.
Hence using dominated convergence we obtain that 
\begin{equation}\label{ouicaconv}
 \lim_{\gep,\gep'\to 0} \bbE\left[ M_{\gep,q }\bar M_{\gep',q } \right]= \int_{\cD^2} e^{(\alpha^2+\beta^2) K(x,y)}\bbP[ \bar A_{q}(x,y)]\dd x \dd y<\infty.
\end{equation}
\qed

\begin{proof}[Proof of Lemma \ref{ladomine}]
 The change of measure given by \eqref{cameron} is simply a Cameron-Martin shift.
 It does not change the variance of the field $X_{k}$ but it modifies its mean,
 we have 
 \begin{equation}
  \bbE_{\gep,\gep',x,y}[X_{k}(z)]=\alpha\left( K_{\gep_k,\gep}(z,x)+ K_{\gep_k,\gep'}(z,y)\right)=: \alpha J_{\gep,\gep'}(k,z).
 \end{equation}
Hence we have 
\begin{equation}\label{scrime}
 \tilde \bbP_{\gep,\gep',x,y}(A_q(x,y))=\bbP\left[ \forall k\ge q, \forall z\in \{x,y\}, \ X_k(z)\le  k\gl - \alpha J_{\gep,\gep'}(k,z) \right].
\end{equation}
To obtain the domination \eqref{ladomi} it is sufficient to evaluate the probability 
\begin{equation}\label{probk0}
\bbP\left[ X_{k_0}(x) \le  k_0\gl - \alpha J_{\gep,\gep'}(k_0,z) \right],
\end{equation}
with $k_0(\gep,x,y):=\log \left(\frac{1}{|x-y|\vee \gep}\right)$.
We have from \eqref{stimatz} for some adequate constant $C$
\begin{equation}\begin{split}\label{wiizz}
J_{\gep,\gep'}(k,x)&\ge 2k_0-C/\alpha,\\
\Var(X_{k_0}(x))&\le  k_0+C,
\end{split}
\end{equation}
Assuming that  $k_0(\gl-2\alpha) + C$ is negative and that $k_0\ge q$
(which we can, all other cases can be treated by taking $C_q$ large since a  probability is always smaller than one)
 the probability \eqref{probk0} is smaller than 
\begin{equation}
 \bbP\left[  X_{k_0} (x)\le  k_0(\gl-2\alpha)+C \right] \le  2 e^{-\frac{(k_0(2\alpha-\gl)-C)^2}{2(k_0+C)}}\le C' (|x-y|\vee \gep)^{ \frac{(2\alpha-\gl)^2}{2} }  
\end{equation}
where we have used \eqref{wiizz} and  the following simple Gaussian bound valid for all $u\ge 0$  
\begin{equation}\label{gbound}
 \frac{1}{\sqrt{2\pi} \sigma}\int^{\infty}_u e^{-\frac{x^2}{2\sigma^2}}\dd x \le 2 e^{-\frac{u^2}{\sigma^2}}.
\end{equation}
The convergence for fixed distinct values of $x$ and $y$ is simply a consequence of the convergence
of  $J_{\gep,\gep'}(k,x)$ and $J_{\gep,\gep'}(k,y)$ to $H_k(x,y)$ and $H_k(y,x)$ respectively. Some care is needed here since we are dealing with 
countably many $X_k$'s.
Let us define for any integer $\ell\ge q$
\begin{equation}\begin{split}
B^\ell_q(\gep,\gep')&:=\left\{ \forall k\in \lint q,   \ell\rint ,\forall z\in\{x,y\},  \ X_k(z)\le   k\gl - \alpha J_{\gep,\gep'}(k,z)\right\},\\
C^\ell_q(\gep,\gep')&:= \left\{ \exists k \ge \ell+1,  \exists z\in\{x,y\}, \  X_k(z)>   k\gl - \alpha J_{\gep,\gep'}(k,z)\right\}.
 \end{split}
\end{equation}
We use the notation $B^\ell_q(0)$  for the event corresponding to $\gep, \gep'=0$.
We have from \eqref{scrime}
\begin{equation}
  \tilde \bbP_{\gep,\gep',x,y}(A_q(x,y))= \bbP[B^\ell_q(\gep,\gep')]- \bbP[B^\ell_q(\gep,\gep')\cap C^\ell_q(\gep,\gep')].
\end{equation}
Hence we have
\begin{multline}
   |\tilde \bbP_{\gep,\gep',x,y}(A_q(x,y))- \bbP[ \bar A_q(x,y)]| 
 \\  \le \big|\bbP[B^{\ell}_q(\gep,\gep')]-\bbP[B^{\ell}_q(0)]\big|+  \big|\bbP[B^{\ell}_q(0)]-\bbP[ \bar A_q(x,y)]\big|+ \bbP[C^\ell_q(\gep,\gep')].
\end{multline}
Let us fix $\delta>0$. We are first going to show that for $\ell=\ell_0(\delta,x,y)$ sufficiently large, each of the two last terms are smaller than $\delta/3$, and then conclude using the fact that since for a fixed $\ell_0$ we have
$$ \lim_{\gep,\gep'\to 0}\bbP\left[ B^{\ell_0}_q(\gep,\gep')\right]=\bbP[B^{\ell_0}_q(0)],$$
so that the first term can also be made smaller than $\delta/3$ by choosing $\gep$ and $\gep'$ small.
Since $\cap_{\ell\ge q} B^\ell_q(0)=\bar A_q(x,y)$, the second term is indeed small if $\ell_0$  sufficiently large. Now
from \eqref{stimatz} we have for every $\gep,\gep'$ and $z\in\{x,y\}$
$$J_{\gep,\gep'}(k,z)\le k+\log\frac{1}{|x-y|}+C. $$
Using the Gaussian bound \eqref{gbound} and making the value of $\ell_0$ large if necessary, this implies that for some constant $C'$ (allowed to depend on $x$ and $y$)
$$ \bbP[ C^{\ell_0}_q(\gep,\gep')]\le \bbP[\exists k \ge \ell_0+1, \exists z\in\{y,z\}, X_k(z)>   k(\gl-\alpha)- C' ].$$
The above probability can be bounded from above by something arbitrarily small if $\ell_0$ is large by using a union bound and the Gaussian tail bound \eqref{gbound} (here we are using that $\alpha<\gl$ and the fact that the variance of $X_k$ is of order $k$).

\end{proof}

\subsection{The limit does not depend on $\theta$}\label{nodepend}

Given $\theta'$ another smoothing kernel we let $X'_{\gep}$ be the regularized field obtained by convolution with $\theta'_{\gep}$ and $M_{q,\gep}'$ be the corresponding truncated partition function (based on the event $A'_{q,\lambda}$ defined as in \eqref{defaqgl} with $X$ replaced by $X'$).
We show that $\lim_{\gep\to 0}\bbE[|M_{q,\gep}- M_{q,\gep}'|^2]=0$  by showing that 
\begin{multline}
\lim_{\gep\to 0}\bbE[|M_{q,\gep}|^2]=\lim_{\gep\to 0} \bbE[|M'_{q,\gep}|^2]
=\lim_{\gep\to 0}\bbE[ M_{q,\gep} \bar M'_{q,\gep}]\\=\int_{D^2} e^{(\alpha^2+\beta^2) K(x,y)}\bbP[ \bar A_{q}(x,y)]f(x)f(y)\dd x \dd y.
\end{multline}
The two first convergence statements are special cases of \eqref{ouicaconv}.
For $\bbE[ M_{q,\gep} \bar M'_{q,\gep}]$, we just have to prove a variant of Lemma \ref{ladomine} for the adequate tilting measure, which can be done without difficulty by reproducing the exact same proof.

\medskip

\noindent \textit{About Remark \ref{wopopop}}
In \cite{LRV15}, instead of being approximated by convolutions, $X$ is given a martingale approximation (see \cite[Equation (2.2)]{LRV15}) which we denote here by $\tilde X_{\gep}$.
If similarly to what is done above, we replace $M_{q,\gep}$ by $\tilde M_{q,\gep}$ which is defined by replacing $X$ by $\tilde X$ in every definition, we can also prove in the same manner (and under the assumption of regularity given in \cite{LRV15} for the covariance kernel of $\tilde X_{\gep}$) that
 $$\lim_{\gep\to 0}\bbE[|M_{q,\gep}- \tilde M_{q,\gep}|^2]=0,$$
 and hence that our limit coincides with the chaos defined in \cite{LRV15}.

\section{Proof of Theorem \ref{decoco}} \label{complexgamma}

To prove the main result of this section, we  partly adapt the strategy used in Section \ref{secdecoco}, but we need to introduce some refinement to it because the real and imaginary part of the field cannot be treated separately anymore. Our method requires some additional technical assumption on the covariance function, which ensures that the field $X$ can be written as a sum of independent functional increments. 
It has been recently proved in \cite{junnila2019} that this decomposition assumption is satisfied locally as soon as our function $L$ in \eqref{lexpress} is sufficiently regular.

\subsection{The result for kernels admitting a nice decomposition}

We are going to prove the result with an additional assumption on the covariance kernel. We assume that $K$ can be written in the form
\begin{equation}\label{tiltz}
 K(x,y)=  Q_0(x,y)+\sum_{n\ge 1}^{\infty} Q_n(x,y),
\end{equation}
where $Q_0(x,y)$ is  positive definite (in the sense \eqref{treks}) and H\"older continuous (in both variable $x$ and $y$).
The functions 
$(Q_n)_{n\ge 1}$ are continuous positive definite function on $\cD$ satisfying 
\begin{equation}\begin{cases}\label{toltz}
Q_n(x,y)\ge 0 ,\\  Q_n(x,x)=1, \\
Q_n(x,y)=0 \quad \text{ if } |x-y|\ge e^{-n}, 
\\ |Q_n(x,y)-Q_n(x',y')|\le
Ce^n(|x-x'|+|y-y'|) 
\end{cases}\end{equation}
for every $x,x',y,y' \in \cD$.
It is not  difficult to check that these assumptions imply in particular that \eqref{tiltz} hold.
Our main task is to prove convergence of $M^{(\gamma)}_{\gep}$ in this setup.
\begin{proposition}\label{avecladeco}
 Let us assume that $K$ satisfies assumption \eqref{tiltz}- \eqref{toltz}.
 If $(\alpha,\beta)\in \mathcal P_{\mathrm{sub}}$, $\gamma=\alpha+i\beta$ and 
$M^{(\gamma)}_{\gep}(f)$ is defined as in \eqref{lesssimple}, then the following limit exists in probability and in $\bbL_1$
 
 \begin{equation}
  \lim_{\gep\to 0} M^{(\gamma)}_{\gep}(f)=M^{(\gamma)}_0(f).
 \end{equation}
 Furthermore the limit does not depend on the choice of  the smoothing kernel $\theta$.
\end{proposition}

\begin{rem}\label{laremark}
 Our assumptions on $Q$ are not all necessary. For instance the assumptions $Q_n(x,y)\ge 0 $ could be suppressed. Some mild assumptions on the decay of correlation could  replace the one about compact support and $Q_n(x,x)=1$ could be replaced by $|Q_n(x,x)-1|\le r(n)$ for a summable function $n$. As we felt that this would not present a significant extension of Theorem \ref{coco} in any case, we preferred to keep stronger assumptions in order to keep the proof as readable as possible.
 \end{rem}

\subsection{Deducing Theorem  \ref{coco} from the decomposable case}

To prove Theorem \ref{coco} building on the case of decomposable kernels, we crucially rely on a result in \cite{junnila2019} which implies that
if $L\in H^{s}_{\mathrm{loc}}(\cD\times \cD)$, $s>d$, then our kernel $K$ admits a decomposition satisfying \eqref{tiltz}-\eqref{toltz}.

\medskip

\noindent We present only a simple consequence of this result which is sufficient to our purpose.
{\red We start with a function $\kappa: \bbR_+ \to \bbR$, satisfying the following assumptions:
\begin{itemize}
 \item [(i)] $\kappa$ is Lipshitz-continuous and non-negative,
 \item [(ii)] $\kappa(0)=1$, $\kappa(r)=0$ for $r\ge 1$.
 \item [(iii)]$(x,y)\mapsto \kappa(|x-y|)$ defines a positive definite function on $\bbR^d\times \bbR^d$. 
\end{itemize}
One possibility is to define
\begin{equation}\label{defkappa}
\kappa(r):= \frac{|B(0,1)\cap B(2r{\bf e}_1,1)|}{|B(0,1)|}    
\end{equation}
where $B(x,R)$ denote the open Euclidean ball of radius $R$ and ${\bf e}_1$ is the vector $(1,0,\dots,0)$, and $|\cdot|$ is used for Lebesgue measure.}
The following proposition is a particular case of \cite[Theorem
4.5]{junnila2019}.

\begin{proposition}\label{dapopo}
  If $K$ is of the form \eqref{lexpress} with $L\in H^{s}_{\mathrm{loc}}(\cD\times \cD)$, $s>d$, and $\kappa$ is as above,  then for any $z\in \cD$ there exists $\delta(z)>0$ and $t_0(z)>0$ which are such that the function (extended by continuity on the diagonal)
  \begin{equation}
    Q_0(x,y):= L(x,y)-\int^{\infty}_{t_0+1} \kappa (e^{t}|x-y|)\dd t + \log \frac{1}{|x-y|}, 
  \end{equation}
   is a positive definite function on $B(z,\delta(z))$.
 
\end{proposition}

\begin{proof}[Deducing Theorem \ref{coco} from Propositions \ref{avecladeco} and \ref{dapopo}]

Note that from Sobolev and Morrey's inequality, the assumption $L\in H^{s}_{\mathrm{loc}}(\cD \times \cD)$ , $s>d$, implies that  $L(x,y)$ is locally H\"older continuous and thus so is $Q_0$ (the reader can check that $Q_0-L$ is Lipshitz).
Now defining for $n\ge 1$
\begin{equation}
 Q_n(x,y):= \int^{t_0+n+1}_{t_0+n} \kappa (e^{t}|x-y|)\dd t,
\end{equation}
it is easy to check that conditions \eqref{tiltz}-\eqref{toltz} are satisfied on $B(z,\delta(z))$.
Now since the support of $f$ is compact, we can cover it by a finite collection of balls $B(z_i,\delta_i)^k_{i=1}$ obtained with Proposition \ref{dapopo}. {\red Using a partition of unity we can write  $f$ as a sum $f=\sum_{i=1}^k f_i$ where each $f_i$ has a support included in $B(z_i,\delta_i)^k_{i=1}$.
Then  we establish the convergence of  
\begin{equation}
M^{(\gamma)}_{\gep}(f)= \sum_{i=1}^k M^{(\gamma)}_{\gep}(f_i)
\end{equation}
simply using Proposition \ref{avecladeco} for each term of the sum.}

\end{proof}
\subsection{Extending the probability space and truncating the partition function}\label{lasecsec}

To prove Proposition \ref{avecladeco}, we are going to work in an extended probability space. Together with the Gaussian process $X$ indexed by $\cM_K$ (recall Section \ref{leledef})
we define a process $(Y_n(x))_{n\ge 1,x\in \cD}$ such that $(X,Y)$ is jointly Gaussian and centered. The covariance function of $Y$ is given by 
\begin{equation}\label{covarian1}
 \bbE\left[ Y_n(x) Y_m(y) \right]= \sum_{k=1}^{m\wedge n} Q_k(x,y)=:K_{n\wedge m}(x,y),
\end{equation}
and the covariance with $X$ is given by
 $\bbE[ Y_n(x) \langle X,\mu \rangle ]= \int_\cD K_n(x,z) \mu(\dd z),$ for $\mu\in \cM_K$.
In particular we have for $y\in \cD_{\gep}$ (recall \eqref{cdgep})
\begin{equation}\label{covarian2}
 K_{n,\gep}(x,y):=\bbE\left[ Y_n(x) X_\gep(y) \right] =\int_{\cD} K_n(x,z) \theta_\gep(x-z)\dd z.
\end{equation}
We consider for every $n$ a continuous version of the field $Y_n(\cdot)$ (which exists since $K_n$ is Lipshitz).
{\red
The existence (and uniqueness in distribution) of such a Gaussian process $(\langle X, \mu\rangle_{\mu \in \mathcal M_K},(Y_n(x))_{n\ge 1, x\in \cD})$ follows simply from Kolmogorov's extension Theorem after checking that for finite dimensional marginals, the covariance matrices given by \eqref{hatK}, \eqref{covarian1} and \eqref{covarian2} are  positive definite. }

\medskip

We assume (without loss of generality) that both $\alpha$ and $\beta$ are positive, 
that $(\alpha,\beta)\in \cP_{\mathrm{sub}}$ with $\alpha>\sqrt{d/2}$ (the other cases belong to the $\bbL_2$ region), and consider $\lambda$ satisfying \eqref{surlambda}.
We can now introduce a truncated version of the partition function, similar to the one considered in the previous section, but with $X_{e^{-k}}$ replaced by $Y_k$.
We recycle the notation of the previous section, and redefine the events $A_{q,\gl}(x)$, $\cA_{q,\gl}$ by setting 
\begin{equation}\begin{split}\label{defaqq}
                 A_{q,\gl}(x)&:= \left\{  \forall k \ge  q, \ Y_k(x) \le  k \gl  \right\},\\
                  \cA_{q,\gl}(f):= \bigcap_{x\in \Supp(f)}  & A_{q,\gl}(x)= \left\{  \forall k\ge q, \quad  \sup_{x\in \Supp(f)} Y_k(x) \le  k \gl  \right\},
                \end{split}
\end{equation}
and then set 
\begin{equation}
 M^{(\gamma)}_{\gep,q}= \int_{\cD} e^{\gamma X_{\gep}(x)-\frac{\gamma^2}{2}K_{\gep}(x)}\ind_{A_{q,\gl}(x)} f(x)\dd x.
\end{equation}
Then we proceed as in\ref{secstratz}.  Firs we show that $M^{(\gamma)}_{\gep,q}$ converges for every value of $q$. 
\begin{proposition}\label{secondacase}
 For every $q\ge 1$ the sequence  $(M^{(\gamma)}_{\gep,q })_{\gep\in(0,1]}$ is Cauchy in $\bbL_2$. In particular we the existence of the following limit
 
\begin{equation}
 \lim_{\gep\to 0}M^{(\gamma)}_{\gep,q }(f):= M^{(\gamma)}_{0,q }(f).
\end{equation}
Furthermore the limit does not depend on the choice of $\theta$.
  
 \end{proposition}
Then we prove that for large $q$, $M^{(\gamma)}_{\gep,q}$ coincides with $M^{(\gamma)}_{\gep}$ 
with high probability, which is a consequence of the following result.
 
\begin{proposition}\label{lapropo2}
 We have for any $\gl>\sqrt{2d}$ 
 \begin{equation}
  \lim_{q\to \infty} \bbP[\cA_{q,\gl}]=1.
 \end{equation}

\end{proposition}

\subsection{Reducing the proof of Proposition \ref{secondacase} to a domination statement}

As in previous cases, we only need to  prove that $\bbE \left[ M^{(\gamma)}_{\gep,q}(f) \overline{M} ^{(\gamma)}_{\gep',q}(f)\right]$ converges to a finite limit when both $\gep$ and $\gep'$ go to zero (the fact that the limit does not depend on $\theta$ follows from the argument developed in Section \ref{nodepend} which also applies to the present case).
Again we drop the dependence in $\gamma$ and $f$ in the notation. {\red We further always assume that (recall \eqref{cdgep})
\begin{equation}\label{always}
  \Supp(f) \subset\cD_{\gep}.
\end{equation}
For all results  in the remainder of this section, the constant $C$ are allowed to depend on $K$, $\gamma$, $q$, $\Supp(f)$ and $\theta$, but not on $x$, $y$, $\gep$ and $\gep'$.}
Setting $A_q(x,y):= A_{q,\gl}(x)\cap A_{q,\gl}(y)$
and interpreting the real part of the exponential tilt as a change of measure (we use the definition \eqref{cameron} for 
$\tilde \bbP_{\gep,\gep',x,y}$), we have
\begin{equation}\begin{split}\label{letrucaudessus}
 \bbE \left[M_\gep \overline{M} _{\gep'}\right]
 &=\int_{\cD^2} \bbE\left[e^{\gamma X_{\gep}(x)+\bar \gamma X_{\gep'}(y)-
 \frac{\gamma^2 K_{\gep}(x)+\bar \gamma^2  K_{\gep'}(y)}{2}} \ind_{A_q(x,y)}\right]f(x)f(y)\dd x \dd y
 \\&=  \int_{\cD^2}e^{\alpha^2 K_{\gep,\gep'}(x,y)+ \left(
 \frac{\beta^2}{2} -i\alpha\beta\right) K_{\gep}(x)+ \left(\frac{\beta^2}{2} +i\alpha\beta\right) K_{\gep'}(y) } \\
 & \quad \quad \quad \quad  \quad \quad \quad  \quad\times \tilde \bbE_{x,y,\gep,\gep'} \left[ e^{ i\beta \left(X_{\gep}(x)- X_{\gep}(y)\right) }\ind_{A_q(x,y)}\right]f(x)f(y)\dd x \dd y
 \end{split}
\end{equation}
As noticed before, under $\tilde \bbP_{\gep,\gep',x,y}$, the mean of the field is shifted and its covariance is preserved. More precisely we have for any $\eta\in \{\gep,\gep'\}$, $n\ge 1$ and $z\in \cD_{\gep}$ 
\begin{equation}\label{leshift}\begin{split}
 \tilde \bbE_{\gep,\gep',x,y}[X_{\eta}(z)]=\alpha \left( K_{\gep,\eta}(x,z)+K_{\gep',\eta}(x,z)\right),\\
  \tilde \bbE_{\gep,\gep',x,y}[Y_n(z)]=\alpha \left(K_{n,\gep}(z,x)+K_{n,\gep'}(z,y)\right).
  \end{split}
\end{equation}
We introduce the functions $L_{n,\gep,\gep'}$ and $L_n$ defined by 
(these functions depend also on $x$ and $y$ but we want to keep the notation as light as possible)
\begin{equation}\label{defL}
L_{n,\gep,\gep'}(z):= K_{n,\gep}(z,x)+K_{n,\gep'}(z,y) \text{ and } L_n(z)=K_{n}(z,x)+K_{n}(z,y).
\end{equation}
Then setting  
\begin{equation}
 \tilde A_{q,\gep,\gep'}(x,y):= \left\{ \forall n\ge q,  \forall z\in \{x,y\},\ Y_n(z) \le \gl n-\alpha L_{n,\gep,\gep'}(z) \right\}
\end{equation}
we deduce from \eqref{letrucaudessus} and \eqref{leshift} that 
\begin{equation}\label{lootz}
 \bbE \left[M_\gep \overline{M} _{\gep'}\right]  = \int_{\cD^2} e^{ (\alpha^2+\beta^2) K_{\gep,\gep'}(x,y) }
\bbE\left[ :e^{ i\beta \left(X_{\gep}(x)- X_{\gep'}(y)\right)}: \ind_{\tilde A_{q,\gep,\gep'}(x,y)}\right] 
  f(x)f(y)\dd x \dd y,
\end{equation}
where we have used the Wick exponential notation for a centered Gaussian variable $Z$
\begin{equation}
:e^{u Z}:\ := e^{uZ-\frac{u^2}{2}\bbE[Z^2]}.
\end{equation}
To conclude we use the following result, which is  analogous to Lemma \ref{ladomine} and use dominated convergence.

\begin{proposition}\label{ladomine2}
There exists a constant $C>0$ such that  for  every $\gep'\le \gep$ and   $x,y\in \Supp(f)$, 
\begin{equation}\label{xdomine}
\left|\bbE\left[ : e^{ i\beta \left(X_{\gep}(x)- X_{\gep'}(y)\right)}:\ind_{\tilde A_{q,\gep,\gep'}(x,y)}\right]\right|\le C (|x-y|\vee \gep)^{\frac{(2\alpha-\gl)^2}{2}} 
\end{equation}
Furthermore the  above expectation admits a limit when $\gep$ and $\gep'$ both go to zero.
\end{proposition}

\noindent From \eqref{xdomine} the integrand in \eqref{lootz} is dominated by $C |f(x)f(y)|  |x-y|^{\frac{(2\alpha-\gl)^2}{2}-(\alpha^2+\beta^2)}$ which is integrable given our assumption on the value of $\gl$ \eqref{surlambda}.
The proof of Proposition \ref{ladomine2} is slightly more involved than that of Lemma \ref{ladomine} and requires a new method. We develop it in the following subsection.

\subsection{Proof of Proposition \ref{ladomine2}}

Our main idea is to decompose $\ind_{\tilde A_{q,\gep,\gep'}(x,y)}$ into an algebraic sum of indicator functions of events 
in $\cF_n$ for some finite $n$ where
$\cF_n:= \sigma( Y_k(\cdot), k\le n)$.
To underline the advantage of dealing with events in  $\cF_n$, let us perform a few Gaussian computations.
Note that we have for $\eta\in \{ \gep,\gep'\}$ and $z\in \cD_{\eta}$
\begin{equation}\label{laconvolzz}
 \bbE \left[X_{\eta}(z) \ | \ \cF_n \right]= \int_{\cD} \theta_\eta(z-z_1)Y_n(z_1) \dd z_1 =: Y_{n,\eta}(z)
\end{equation}
hence using the fact that in a Gaussian space the conditional expectation of the Wick exponential coincides with the Wick exponential of the conditional expectation, if  $B_{n,\gep,\gep'}\in \cF_n$  we have
\begin{equation}\label{lesmezur}
\bbE\left[ :e^{ i\beta \left(X_{\gep}(x)- X_{\gep'}(y)\right)}: \ind_{B_{n,\gep,\gep'}}\right]=\bbE\left[ :e^{ i\beta \left(Y_{n,\gep}(x)- Y_{n,\gep'}(y)\right)}:  \ind_{B_{n,\gep,\gep'}}\right],
\end{equation}
and since $\lim_{\gep \to 0} Y_{n,\gep}=Y_n$ the convergence of the right hand side can be proved using dominations argument provided $B_{n,\gep,\gep'}$ is suitably chosen.
We set 
\begin{equation}
n_0(\gep,x,y):=\left\lceil \log \frac{1}{|x-y|\vee \gep}\right\rceil \quad \text{ and } \quad n^{\star}_0(x,y):= \left\lceil \log \frac{1}{|x-y|}\right\rceil.
\end{equation}
We let $A^{(\gep,\gep')}_{n_0}$ denote the event that the upper bound constraint in $\tilde A_{q,\gep,\gep'}$ is satisfied for all $n\le n_0$ 
\begin{equation}
A^{(\gep,\gep')}_{n_0}:= \left\{ \forall n \in \lint q,n_0\rint,
\forall z\in \{x,y \}, \  Y_n(z)\le n \gl -\alpha L_{n,\gep,\gep'}(z)  \right\}.
\end{equation}
Now for $n\ge n_0+1$, we define  $B^{(\gep,\gep')}_{n,1}$ (resp.  $B^{(\gep,\gep')}_{n,2}$) 
the events that $A^{(\gep,\gep')}_{n_0}$ is satisfied and that 
$n$ is the first index for which  $Y_n(x)$ (resp. $Y_n(y)$) violates the upper constraint in $\tilde A_{q,\gep,\gep'}$
\begin{equation}\begin{split}
 B^{(\gep,\gep')}_{n,1}&:=A^{(\gep,\gep')}_{n_0}\cap\left\{  \inf\{  \ m\ge n_0 \ : \   Y_m(x)> m \gl -\alpha L_{n,\gep,\gep'}(x) \} =n  \right\},\\
  B^{(\gep,\gep')}_{n,2}&:=A^{(\gep,\gep')}_{n_0}\cap\left\{  \inf\{  \ m\ge n_0 \ : \   Y_m(y)> m \gl -\alpha L_{n,\gep,\gep'}(y) \} =n  \right\}.
\end{split}\end{equation}
Finally we define $C^{(\gep,\gep')}_{n,m}:=    B^{(\gep,\gep')}_{n,1}\cap B^{(\gep,\gep')}_{m,2}.$ 
The reader can then quickly check that 
\begin{equation}\label{ladecompa}
 \ind_{\tilde A_{q,\gep,\gep'}}
 = \ind_{A^{(\gep,\gep')}_{n_0}}- \sum_{n\ge  n_{0}+1}\left( \ind_{B^{(\gep,\gep')}_{n,1}}+   \ind_{B^{(\gep,\gep')}_{n,2}}\right) + 
 \sum_{n,m\ge  n_{0}+1}\ind_{C^{(\gep,\gep')}_{n,m}}.
\end{equation}
Note that the events remain well defined in the limit when $\gep, \gep'$ tend to $0$. We let 
$A^{\star}_{n^{\star}_0}$,  $B^{\star}_{n,j}$ and $C^{\star}_{n,m}$ (for $n, m\ge n^{\star}_0+1$) denote the event obtained in the $\gep,\gep'\to 0$ limit, replacing $n_0$ by $n^{\star}_0$ and $L_{n,\gep,\gep'}$ by $L_n$.
We are going to deduce Proposition \ref{ladomine2} from the following estimates.

 \begin{proposition}\label{lezestim}
The following statement holds for a sufficiently large constant $C$ 
  \begin{itemize}
   \item [(A)] We have 
   \begin{equation}\label{le1}
   \left|  \bbE\left[ :e^{ i\beta \left(X_{\gep}(x)- X_{\gep'}(y)\right) }:\ind_{A^{(\gep,\gep')}_{n_0}}\right] \right|\le C (|x-y|\vee \gep)^{\frac{\left(2\alpha-\gl\right)^2}{2}}.
   \end{equation}
and 
\begin{equation}\label{le11}
 \lim_{\gep,\gep' \to 0} \bbE\left[ :e^{ i\beta \left(X_{\gep}(x)- X_{\gep'}(y)\right) }:\ind_{A^{(\gep,\gep')}_{n^\star_0}}\right]
 =  \bbE\left[ :e^{ i\beta \left(Y_{n^\star_0}(x)- Y_{n^\star_0}(y)\right) }:\ind_{A^{\star}_{n^\star_0}} \right].
\end{equation}

 \item [(B)] We have for every $n\ge n_0+1$ and $j=1,2$
   \begin{equation}\label{le2}
   \left|  \bbE\left[ : e^{ i\beta \left(X_{\gep}(x)- X_{\gep}(y)\right) }:\ind_{B^{(\gep,\gep')}_{n,j}}\right] \right|\\   \le  C (|x-y|\vee \gep)^{\frac{(2\alpha-\gl)^2}{2}}e^{\frac{1}{2}[(\gl-\alpha)^2+\beta^2](n-n_0)}.
   \end{equation}
   and for $n\ge n^\star_0+1$ we have
   \begin{equation}\label{le22}
 \lim_{\gep,\gep' \to 0} \bbE\left[ :e^{ i\beta \left(X_{\gep}(x)- X_{\gep'}(y)\right) }:\ind_{B^{(\gep,\gep')}_{n,j}}\right]
 =  \bbE\left[ :e^{ i\beta \left(Y_{n}(x)- Y_{n}(y)\right) }:\ind_{B^{\star}_{n,j}} \right].
\end{equation}

 \item [(C)] We have for every $n,m$  

   \begin{equation}\label{le3}
   \left|  \bbE\left[ :e^{ i\beta \left(X_{\gep}(x)- X_{\gep'}(y)\right) }:\ind_{C_{n,m}}\right] \right|
   \le  C (|x-y|\vee \gep)^{\frac{1}{2}\left(2\alpha-\gl\right)^2}e^{\frac{1}{2}[(\gl-\alpha)^2+\beta^2](n\vee m-n_0)}.
   \end{equation}
    and for $n,m \ge n^\star_0+1$ we have
    \begin{equation}\label{le33}
 \lim_{\gep,\gep' \to 0} \bbE\left[ :e^{ i\beta \left(X_{\gep}(x)- X_{\gep'}(y)\right) }:\ind_{C^{(\gep,\gep')}_{n,m}}\right]
 =  \bbE\left[ :e^{ i\beta \left(Y_{n\vee m}(x)- Y_{n\vee m}(y)\right) }:\ind_{C^{\star}_{n,m}} \right].
 \end{equation}
  \end{itemize}
  \end{proposition}

 \begin{proof}[Proof of Proposition \ref{ladomine2}]
  Looking at \eqref{ladecompa} and using the triangle inequality we deduce from \eqref{le1},\eqref{le2} and \eqref{le3} that
  \begin{multline}
   \left|\bbE\left[ : e^{ i\beta \left(X_{\gep}(x)- X_{\gep}(y)\right)}:\ind_{\tilde A_{q,\gep,\gep'}(x,y)}\right]\right|\\
   \le C
   (|x-y|\vee \gep)^{\frac{1}{2}\left(2\alpha-\gl\right)^2}
   \left(1+  \!\!\ \!\! \!\!\sum_{n\ge n_0+1} \!\! \!\! e^{\frac{(\gl-\alpha)^2+\beta^2}{2}(n-n_0)} +  \!\! \!\! \!\!\sum_{n,m\ge n_0+1}  \!\! \!\! \!\! e^{\frac{(\gl-\alpha)^2+\beta^2}{2}(n+m-2n_0)}\right).
  \end{multline}
Moreover, by dominated convergence (for the sum  in $n$ and $m$) and the convergence results \eqref{le11},\eqref{le22} and \eqref{le33} we have 
\begin{multline}
  \lim_{\gep,\gep' \to 0}\bbE\left[ : e^{ i\beta \left(X_{\gep}(x)- X_{\gep}(y)\right)}:\ind_{\tilde A_{q,\gep,\gep'}(x,y)}\right] \\ =
  \bbE\left[ :e^{ i\beta \left(Y_{n^\star_0}(x)- Y_{n^\star_0}(y)\right) }:\ind_{A^{\star}_{n^\star_0}} \right]+ \sumtwo{n\ge n_0+1}{j=1,2} \bbE\left[ :e^{ i\beta \left(Y_{n}(x)- Y_{n}(y)\right) }:\ind_{B^{\star}_{n,j}} \right]\\+ \sum_{n,m\ge  n_{0}+1} \bbE\left[ :e^{ i\beta \left(Y_{n\vee m}(x)- Y_{n\vee m}(y)\right) }:\ind_{C^{\star}_{n,m}} \right].
\end{multline}

 \end{proof}

 \subsection{Proof of Proposition \ref{lezestim}}
 
 The proof of convergence statement is the easier part and is identical for \eqref{le11},\eqref{le22}, \eqref{le33}(recall  that $n_0=n^\star_0$ when $\gep$ and $\gep'$ are sufficiently small).
 Let us fully detail the case \eqref{le33} for the sake of completeness. Since the event $C^{(\gep,\gep')}_{n,m}$ is $\cF_{n,m}$ measurable we have from \eqref{lesmezur} 
 \begin{equation}
  \bbE\left[ :e^{ i\beta \left(X_{\gep}(x)- X_{\gep}(y)\right) }:\ind_{C^{(\gep,\gep')}_{n,m}}\right]=
    \bbE\left[ :e^{ i\beta \left(Y_{n,\gep}(x)- Y_{n,\gep'}(y)\right) }:\ind_{C^{(\gep,\gep')}_{n,m}}\right].
 \end{equation}
and we can conclude (using dominated convergence) by observing that the quantity inside the expectation converges in probability towards 
$:e^{ i\beta \left(Y_{n}(x)- Y_{n}(y)\right) }:\ind_{C^\star_{n,m}}$
 and is bounded above by $e^{\frac{\beta^2}{2}(K_{n,\gep}(x)+K_{n,\gep}(y)-K_{n,\gep}(x,y)}\le  e^{\beta^2n}$ {\red (from the definitions, \eqref{covarian2}, we have for all $x$ and $y$, $K_{n,\gep}(x)\le n$ for all $x$ and $K_n(x,y)\ge 0$).}

\medskip

To prove the domination part, we are going to rely on the following probability estimates for the events involved in the expectation. The proof of these estimates is postponed to the end of the section.

\begin{lemma}\label{lescle}
There exists a positive constant $C$ such that
the following inequalites are valid for all $x,y\in \Supp(f)$, $\gep'\le \gep$, $n\in \lint n_0+1, \lfloor \log 1/\gep \rfloor \rint $ and
$m\in  \lint  n_0+1, \lfloor \log 1/\gep \rfloor \rint$:
 \begin{equation}\label{cleA}
  \bbP[A^{(\gep,\gep')}_{n_0}]\le C (|x-y|\vee \gep)^{\frac{1}{2}\left(2\alpha-\gl\right)^2}.
 \end{equation}
\begin{equation}\label{cleB}
   \bbP[B^{(\gep,\gep')}_{n,j}]\le C (|x-y|\vee \gep)^{\frac{1}{2}\left(2\alpha-\gl \right)^2} 
   e^{-\frac{(n-n_0)}{2}(\gl-\alpha)^2},
\end{equation}
  \begin{equation}\label{cleC}
   \bbP[C^{(\gep,\gep')}_{n,m}]\le \begin{cases}
                                 C |x-y|^{\frac{1}{2}\left(2\alpha-\gl \right)^2} 
   e^{-\frac{(n+ m-2n_0)}{2}(\gl-\alpha)^2},\quad  &\text{ if } |x-y|\le \gep,\\
   C \gep^{\frac{1}{2}\left(2\alpha-\gl \right)^2} 
   e^{-\frac{(n\vee m-n_0)}{2}(\gl-\alpha)^2}\quad  &\text{ if } |x-y|> \gep.
                               \end{cases}
\end{equation}
\end{lemma}
\noindent We are also going to rely on an estimate for the covariance of $Y_{n,\gep}$.
We set 
$$ K_{n,\gep,\gep'}(x,y):= \bbE[Y_{n,\gep}(x)Y_{n,\gep'}(y)].$$
The following estimate  follows from assumption \eqref{toltz} (we include a proof in Appendix \ref{tokz} for completeness).
\begin{lemma}\label{lemztim}
There exists a positive constant $C>0$ such that 
 for any $x,y\in D$ any $n\ge 1$ if $\gep'\le \gep$ we have
 \begin{equation}\label{stimatz2}
 \left| K_{n,\gep,\gep'}(x,y)- \min\left( \log \frac{1}{|x-y|},\log \frac{1}{\gep} , n \right)\right|\le C.
 \end{equation}
\end{lemma}
 \noindent We now have all the ingredients to prove the domination statements

 \begin{proof}[Proof of \eqref{le1}]
 By Jensen's inequality, we have
 \begin{equation}
 \left| \bbE\left[ :e^{ i\beta \left(Y_{n_0,\gep}(x)- Y_{n_0,\gep'}(y)\right) }:\ind_{A^{(\gep,\gep')}_{n_0}}\right]\right|\le e^{\frac{\beta^2}{2}\Var(Y_{n_0,\gep}(x)- Y_{n_0,\gep'}(y))} \bbP[A^{(\gep,\gep')}_{n_0}].
  \end{equation}
As a consequence of \eqref{stimatz2}, and of the choice for $n_0$, the variance 
$$\Var(Y_{n_0,\gep}(x)- Y_{n_0,\gep'}(y))=K_{n_0,\gep,\gep}(x,x)+ K_{n_0,\gep',\gep'}(y,y)- 2 K_{n_0,\gep,\gep'}(x, y)$$
is uniformly bounded in $x,y,\gep$ and $\gep'$ and we can conclude using \eqref{cleA}.
 \end{proof}

 \begin{proof}[Proof of \eqref{le2} and \eqref{le3}]
  
   The idea is the same for \eqref{le2} and \eqref{le3}. We treat only the latter, which is the more delicate, in details.
The inequality we prove differs according to the value of  $\gep$.  When $|x-y|>\gep$ we prove \eqref{le3} while if $|x-y|\le \gep$ we prove the stricter inequality
  \begin{equation}\label{le31}
   \left|  \bbE\left[ :e^{ i\beta \left(X_{\gep}(x)- X_{\gep'}(y)\right) }:\ind_{C_{n,m}}\right] \right|
   \le  C |x-y|^{\frac{1}{2}\left(2\alpha-\gl\right)^2}e^{-\frac{1}{2}[(\gl-\alpha)^2+\beta^2](n + m-2n_0)}
   \end{equation}
   The reader can check here that simply repeating the proof of \eqref{le1} replacing $n_0$ by  $n\vee m$ (case $(C)$), does not yield a satisfactory result (we obtain a factor $\beta^2$ instead of the desired $\beta^2/2$ in the exponential).
We need thus some refinement in the conditioning.
   For simplicity,  let us assume that $n\le m$ (strictly speaking, since we already assumed $\gep\le \gep'$, there is a loss of generality here but this is of no consequence).
We define the $\sigma$-algebra $\cG_{n,m}$ as 
\begin{equation}
\mathcal G_{n,m}= \mathcal G_{n,m}(x,y):= \cF_n \vee \sigma \left( Y_l(y), l\in \lint n+1,m\rint \right).
 \end{equation}
Clearly we have  $C^{(\gep,\gep')}_{n,m}\in \mathcal G_{n,m}$.
Hence similarly to \eqref{lesmezur} we have 
\begin{multline}
  \left|\bbE\left[ :e^{ i\beta \left(X_{\gep}(x)- X_{\gep'}(y)\right) }:\ind_{C^{(\gep,\gep')}_{n,m}}\right]\right|
  = \left|\bbE \left[ :e^{ i\beta \left(\bbE\left[ X_{\gep}(x)- X_{\gep'}(y) \ | \ \mathcal G_{n,m}\right]\right) }: \ind_{C^{(\gep,\gep')}_{n,m}}\right]\right|
  \\
  \le e^{\frac{\beta^2}{2}\Var \left(\bbE\left[ X_{\gep}(x)- X_{\gep'}(y) \ | \ \mathcal G_{n,m}\right]\right)} \bbP(C^{(\gep,\gep')}_{n,m}).
  \end{multline}
We can conclude using \eqref{cleC}, provided that one can show that 
\begin{equation}\label{tcherto}
 \Var \left(\bbE\left[ X_{\gep}(x)- X_{\gep'}(y) \ | \ \mathcal G_{n,m}\right]\right)\le \begin{cases}
 n+m-2n_0, \quad & \text{ if }  |x-y|\le \gep,\\
 m-n_0, \quad & \text{ if } |x-y|>\gep.
\end{cases}
\end{equation}
We perform a decomposition of  $\bbE\left[ X_{\gep}(x)- X_{\gep'}(y) \ | \ \mathcal G_{n,m}\right]$ into a sum of orthogonal Gaussian variables.
We let $Z_n:=Y_n-Y_{n-1}$ denote the $n$-th increment of $Y_n$. Using independence of the increments we obtain
\begin{equation}\label{lagasssse}
 \bbE\left[ X_{\gep}(x)- X_{\gep'}(y) \ | \ \mathcal G_{n,m}\right]
 =  Y_{n,x}(x)- Y_{n,\gep'}(y)+\sum^m_{k=n+1} \bbE\left[ X_{\gep}(x)- X_{\gep'}(y) \ | \ Z_k(y)\right].
\end{equation}
We have 
\begin{equation}\label{primpart}
  \Var(Y_{n,x}(x)- Y_{n,\gep'}(y))=K_{n,\gep,\gep}(x,x)+ K_{n,\gep',\gep'}(y,y)-2 K_{n,\gep,\gep'}(x,y)
\end{equation}
and thus, as a consequence of \eqref{stimatz2} we have
\begin{equation}\label{lezca}
   \Var(Y_{n,\gep}(x)- Y_{n,\gep'}(y))\le \begin{cases}
                                        n-n_0+C \quad & \text{ if } |x-y|> \gep,\\
                                        2(n-n_0)+C & \text{ if } |x-y|\le\gep.
                                       \end{cases}
\end{equation}
On the other hand, a simple Gaussian computation yields
\begin{equation}\label{lagassse}
 \Var \left(\bbE\left[ X_{\gep}(x)- X_{\gep'}(y) \ | \ Z_k(y)\right]\right)
 = \bbE\left[ (X_{\gep}(x)- X_{\gep'}(y))Z_k(y) \right]^2 \bbE[ Z^2_k(y)]^{-1}.
\end{equation}
Similarly to \eqref{laconvolzz} we have for $\eta\in \{ \gep,\gep'\}$ and $z\in D$,  $$\bbE\left[ (X_{\eta}(z) \ | \ (Z_k(z'))_{z'\in \cD} \right]=
\int \theta_{\eta}(z_1-z)Z_k(z_1)\dd z_1).$$
This allow to compute the covariance and we have
\begin{equation}
 \bbE\left[ (X_{\gep'}(y)-X_{\gep}(x))Z_k(y) \right]= \int_{\cD} \left( \theta_{\gep'}(y-z)-
\theta_\gep(x-z)\right) Q_k(z,y) \dd z.
\end{equation}
Now since $0\le Q_k(z,y)\le 1$ we have  
\begin{equation}
\left|\int_{\cD} \left( \theta_{\gep'}(y-z)-
\theta_\gep(x-z)\right) Q_k(z,y) \dd z \right| \le 1.
\end{equation}
and thus (recall that $\bbE[ Z^2_k(y)]=Q_k(y,y)=1]$)
\begin{equation}
 \Var \left( \sum^m_{k=n+1} \bbE\left[ X_{\gep}(x)- X_{\gep'}(y) \ | \ Z_n(y)\right]\right)\le  m-n,
\end{equation}
which together with \eqref{tcherto} concludes the proof of \eqref{lezca}.
The case \eqref{le2} is dealt with similarly but with a conditioning with respect to $\cG_{n_0,n}(y,x)$ (for $j=1$) or $\cG_{n_0,n}(x,y)$ (for $j=2$).
 \end{proof}

\begin{proof}[Proof of Lemma \ref{lescle}]
The proof of \eqref{cleA} is identical to that of \eqref{ladomi} in Lemma \ref{ladomine}. It is sufficient to observe that 
\begin{equation}
 \bbP\left(A^{(\gep,\gep')}_{n_0}\right)\le \bbP( Y_{n_0}(x)\le \gl n_0 + \alpha L_{n,\gep,\gep'} (x) ) \le 
 \bbP( Y_{n_0}(x)\le (\gl-2\alpha)n_0+C).
\end{equation}
For \eqref{cleB}-\eqref{cleC} we use the same idea and restrict the event to a single inequality.
Let us give the details for \eqref{cleC}, the case \eqref{cleB} being similar but simpler.
We assume here also for simplicity that $m\ge n$. Let us start with the case $|x-y|\le \gep$.
Note that if $C^{(\gep,\gep')}_{n,m}$ is satisfied then we have 
\begin{multline}\label{laconseq}
 Y_{n_0,n}(x)+Y_{n_0,m}(y)\\
 > \gl(n+m- 2 n_0)-\alpha[ (L_{n,\gep,\gep'}-L_{n_0,\gep,\gep'}) (x)+(L_{m,\gep,\gep'}-L_{n_0,\gep,\gep'}) (y)].
\end{multline}
where we used the short-hand notation $Y_{n_1,n_2}:=Y_{n_2}-Y_{n_1}$.
If we let $\hat C^{(\gep,\gep')}_{n,m}$ denote the event in  \eqref{laconseq}, as $\hat C^{(\gep,\gep')}_{n,m}$ is independent from $\cF_{n_0}$ and hence of $A^{(\gep,\gep')}_{n_0}$ (since  $Y_{n_0,n}$ and $Y_{n_0,m}$ are), with the bound already proved for $A^{(\gep,\gep')}_{n_0}$, we only need to show that 
\begin{equation}\label{zeultimate}
 \bbP(\hat C^{(\gep,\gep')}_{n,m}) \le e^{-\frac{(\alpha-\gl)^2}{2}(n+m-2n_0)}.
\end{equation}
Hence we need an upper bound on the variance of  $Y_{n_0,n}(x)+Y_{n_0,m}(y)$ and on $ (L_{n,\gep,\gep'}-L_{n_0,\gep,\gep'}) (x)+(L_{m,\gep,\gep'}-L_{n_0,\gep,\gep'}) (y)$.
We have 
\begin{equation}\label{tadz}
\begin{split}
  \Var\left( Y_{n_0,n}(x)+Y_{n_0,m}(y)\right)&= n+m-2n_0,\\
  (L_{n,\gep,\gep'}-L_{n_0,\gep,\gep'}) (x)+(L_{m,\gep,\gep'}-L_{n_0,\gep,\gep'}) (y)& \le n+m-2n_0+C,
 \end{split}
\end{equation}
where the first line comes from the fact $Y_{n_0,n}(x)$ and $Y_{n_0,m}(y)$ are independent {\red with respective variance $n-n_0$ and $m-n_0$} (due to Assumption \eqref{toltz} and the fact that $|x-y|\le \gep$). The second line comes from Lemma \ref{lemztim}.
Then \eqref{zeultimate} is a consequence of \eqref{gbound} and \eqref{tadz}.

\medskip

\noindent When $|x-y|\le \gep$ we observe that $C^{(\gep,\gep')}_{n,m}$ implies
\begin{equation}
 Y_{n_0,m}(y)> \gl(m- n_0)-(L_{m,\gep,\gep'}-L_{n_0,\gep,\gep'}) (y),
\end{equation}
and we conclude similarly using  \eqref{gbound} together with the following estimates
\begin{equation}
 \begin{split}
  \Var\left(Y_{n_0,m}(y)\right)&= m-n_0,\\
 (L_{m,\gep,\gep'}-L_{n_0,\gep,\gep'}) (y)& \le m-n_0+C.
 \end{split}
\end{equation}

\end{proof}

{\red 
\section{Proof of Theorem \ref{asadistrib}}\label{tightness}

To prove that $(M^{\gamma}_{\gep}(\cdot))_{\gep\in (0,1]}$ converges in probability for the $H^{-u}_{\mathrm{loc}}(\cD)$ topology, we need to check that for any $\rho\in \cC^{\infty}_c(\cD)$,  
$(M^{\gamma}_{\gep}(\rho \  \cdot))_{\gep\in (0,1]}$ converges in $H^{-u}(\bbR^d)$ when $\gep$ goes to $0$.
Using Proposition \eqref{dapopo}, and repeating the argument used to prove Theorem \ref{coco} using a partition of unity we can reduce to the case where the kernel $K$ is of the form \eqref{tiltz}. We are going to work on the extended probability space described at the beginning of Section \ref{lasecsec}.

\medskip

Let us consider the Fourier transform of $M^{(\gamma)}_{\gep}(\rho \  \cdot)$ which we denote by $\hat M^{(\gamma,\rho)}_{\gep}(\xi)$   and prove that  $\hat M^{(\gamma,\rho)}_{\gep}$ converges (in probability) in $L^2(\bbR^d, (1+ |\xi|^2)^{-u} \dd \xi )$.
We assume first that $\alpha> \sqrt{d/2}$ - the $\bbL_2$ case is only simpler, we discuss it at the end of the proof. We fix $\gl\in (\sqrt{2d},2\alpha)$ and define a truncated version  $\hat M^{(\gamma,\rho,q)}_{\gep}(\xi)$ of the Fourier transform
 by setting 
 \begin{equation}
  \hat M^{(\gamma,\rho,q)}_{\gep}(\xi):=  \int_{\cD_{\gep}} e^{i\xi.x}e^{\gamma X_\gep(x)-\frac{\gamma^2}{2}K_{\gep}(x)}  \rho(x)\ind_{A_{q,\gl}(x)} \dd x
 \end{equation}
where $ A_{q,\gl}(x)$ is defined as in \eqref{defaqq}. The key point of the proof is the following convergence
 \begin{equation}\label{labarbouze}
  \lim_{\gep,\gep'\to 0}\bbE \left[ \int_{\bbR^{d}}   |\hat M^{(\gamma,\rho,q)}_{\gep}(\xi)-\hat M^{(\gamma,\rho,q)}_{\gep'}(\xi) |^2 (1+|\xi|^2)^{-u} \dd \xi \right]=0.
 \end{equation}
 Before giving a proof of \eqref{labarbouze} (which mostly follows from the work done in Section \ref{complexgamma}) let us show how to use it to conclude.
By completeness of $L^2( \gO\otimes \bbR^d, \bbP\otimes (1+|\xi|^2)^{-u} \dd \xi)$, \eqref{labarbouze} implies that 
 $$\lim_{\gep \to 0} \hat M^{(\gamma,\rho,q)}_{\gep}=: \hat  M^{(\gamma,\rho,q)}_{0}.$$ 
exists $\bbP$-a.s.\ in $L^2( \bbR^d,  (1+|\xi|^2)^{-u} \dd \xi)$. In particular we have 
\begin{equation}\label{laconvarge}
  \lim_{\gep\to 0} \int_{\bbR^{d}}   |\hat M^{(\gamma,\rho,q)}_{\gep}(\xi)-\hat M^{(\gamma,\rho,q)}_{0}(\xi) |^2 (1+|\xi|^2)^{-u} \dd \xi =0.
 \end{equation}
in $\bbL_1$ (and hence in probability). 
Using Proposition \eqref{lapropo2}, applied to $\cA_{q,\gl}(\rho)$, there exists a random value $\mathfrak q\ge 1$ which is such that 
   $\hat M^{(\gamma,\rho,q)}_{\gep}(\xi)= \hat M^{(\gamma,\rho)}_{\gep}(\xi)$ for every $q\ge \mathfrak q$. This together with \eqref{laconvarge} implies that, in the sense of convergence in probability we have
   \begin{equation}
      \lim_{\gep\to 0} \int_{\bbR^{d}}   |\hat M^{(\gamma,\rho)}_{\gep}(\xi)-\hat M^{(\gamma,\rho,\mathfrak q)}_{0}(\xi) |^2 (1+|\xi|^2)^{-u} \dd \xi =0,
   \end{equation}
which is to say that $\hat M^{(\gamma,\rho)}_{\gep}$ converges in probability in $L^2(\bbR^d, (1+ |\xi|^2)^{-u} \dd \xi )$.

\medskip

\noindent It remains to prove \eqref{labarbouze}. We need to check the two following assumptions
\begin{equation}\begin{split}\label{fordo}
 \suptwo{\gep, \gep'\in (0,1]}{ \xi \in \bbR^d}   \bbE \left[ |\hat M^{(\gamma,\rho,q)}_{\gep}(\xi)-\hat M^{(\gamma,\rho,q)}_{\gep'}(\xi) |^2 \right]&< \infty,\\
  \lim_{\gep, \gep'\to 0}   \bbE \left[ |\hat M^{(\gamma,\rho,q)}_{\gep}(\xi)-\hat M^{(\gamma,\rho,q)}_{\gep'}(\xi) |^2 \right]&=0,
\end{split}\end{equation}
and use dominated convergence - here we are using that constants are integrable w.r.t.\ $(1+|\xi|^2)^{-u} \dd \xi$, which is true only when $u\ge d/2$.
The second line of \eqref{fordo} simply corresponds to Proposition \ref{secondacase}.
Concerning the first line of \eqref{fordo}, we simply observe that Equation \eqref{lootz}
provides a bound which is independent of $\xi$, more specifically we have
\begin{multline}
 \bbE \left[  |\hat M^{(\gamma,\rho,q)}_{\gep}(\xi)-\hat M^{(\gamma,\rho,q)}_{\gep'}(\xi) |^2\right]  \\
 \le 4 \int_{\cD^2} e^{ (\alpha^2+\beta^2) K_{\gep,\gep'}(x,y) }
\left|\bbE\left[ :e^{ i\beta \left(X_{\gep}(x)- X_{\gep'}(y)\right)}: \ind_{\tilde A_{q,\gep,\gep'}(x,y)}\right]\right|
\rho(x)\rho(y)\dd x \dd y. 
\end{multline}
Proposition \ref{ladomine2} then ensures that the r.h.s.\ above is bounded uniformly in $\gep$ and $\gep'$.
When $\alpha\le \sqrt{d/2}$ then $|\gamma|^2<d$, and we can apply the same proceedure but without the need to use the truncated version of $\hat M^{(\gamma,\rho,q)}_{\gep}$.

\qed

}
\noindent {\bf Acknowledgement:} The author is grateful to  Nathana\"el Berestycki, R. Rhodes, and Christian Webb for their feedback on a first version of this manuscript. He also wishes to thank the anonymous referee for numerous helpful comments and observation. This work has been realized in part during the author's stay in Aix Marseille University as a MSCA fellow. He acknowledges kind hospitality and support.
This project has received funding from the European Union’s Horizon 2020 research and innovation
programme under the Marie Skłodowska-Curie grant agreement No 837793.

\appendix

\section{Proof of Propositions \ref{lapropo}-\ref{lapropo2}}\label{prooflapropo}

The proof is almost identical for the two results so let us give the full details for the field $X_k$ and explain briefly how to adapt the argument for $Y_k$.
We let $D$ denote the topological support of $f$.
We are going to prove that there exists $c>0$ such that for all $k$ sufficiently large 
\begin{equation}
 \bbP(\cB_k):= \bbP\left(\sup_{x\in D} X_k(x)> \gl k\right)\le \frac{1}{c}e^{-c k}.
\end{equation}
Given $\eta>0$, let us consider a small enlargement of $D$  
$$D^+=D^+_{\eta}:= \{ x\in \cD \ : \ \exists y\in D,\ |x-y|\le \eta \}.$$
We assume that $\eta$ is sufficiently small so that $D^+\subset \cD_{\eta}$ (recall \eqref{cdgep}), and consider in the remainder of the proof that $k\ge \log \frac{1}{\eta}$.

\medskip

Let us fix $\delta>0$ such that  $2d(1+\delta)<\gl^2$.
We first prove give a bound for the maximum on a dyadic grid of mesh $e^{-(1+\delta)k}$ simply by using a union bound. Then we show that local fluctuation within a distance  $e^{-(1+\delta)k}$ are very small in amplitude. This second step of the proof is simply based on a quantitative version of the argument used to prove continuity of Gaussian processes from Kolmogorov-Chentsov criterion (see e.g.\ \cite[Section 2.2]{legallSto} for the classic proof of continuity of Brownian Motion).
Let us consider for $p\ge 1$, $\bbD_p$ the set of points in $D^+$ whose coordinates are integer multiple of $2^{-p}$
 (for large values of $p$ the cardinality of $\bbD_p$ is of order $2^{dp}$) . We set 
 \begin{equation}\label{pk0}
 p^{(k)}_0:= \left\lceil \frac{k(1+\delta)}{\log 2}  \right\rceil.
\end{equation}
 From \eqref{stimatz} the variance of $X_k(x)$ is larger than $k-C$ for every $x\in D_+$ and we have thus from Gaussian tail bound \eqref{gbound}, for some constant $C'$
\begin{equation}
\bbP\left[\max_{x\in \bbD_{p_0}} X_k(x)\ge \gl k-1 \right]\le 2 |\bbD_{p_0}| e^{-\frac{(\gl k-1)^2}{2 (k-C)}}\le C' e^{-\frac{ k [ \gl^2-2d(1+\delta)]}{2}}.
\end{equation}
Note that for every point in $x\in D$ and $p\ge p_0$  there exists a sequence $(x_p)_{p\ge p_0}$, converging to $x$  which satifies
{\red \begin{equation}
\forall p \ge p_0, \quad   x_p\in \bbD_p \quad \text{  and  } \quad   |x_{p+1}-x_{p}|\le \sqrt{d} 2^{-(p+1)}.
\end{equation}}
What we are going to show is that provided that $k$ is sufficiently large, with probability larger that $1-e^{-k}$ we have
\begin{equation}
\forall p \ge p_0+1, \forall y,z \in \bbD_p, \quad   \left\{ |y-z|\le\sqrt{d} 2^{-p} \right\} \   \Rightarrow
 \left\{ |X_{k}(y)- X_{k}(z)|\le \frac{1}{p(p-1)} \right\}
\end{equation}
and we can conclude using continuity that 
$$|X_{k}(x)- X_{k}(x_{p_0})|=\left|\sum_{p\ge p_0+1} X_{k}(x_{p})- X_{k}(x_{p-1})  \right|\le \frac{1}{p_0}\le 1.$$
In order to control local fluctuation, first note that a simple computation allows to deduce from \eqref{labig} that the Lipshitz constant of $K_{\gep}(x,y)$ is at most $C\gep^{-1} |\log \gep|$ (and hence $C ke^{k}$ for $\gep_k$).
Hence the variance of $(X_k(x)-X_k(y))$ is at most $C ke^{k}|x-y|$.
Now taking into account that the number of pair of close-by vertices below is of order $2^p$, we have 
\begin{multline}
 \bbP\left( \maxtwo{x,y \in \bbD_{p}}{|x-y|\le \sqrt{d}2^{-p}} | X_k(x)-X_k(y)|\le \frac{1}{p(p+1)} \right) \\ \le  C 2^{p}
e^{-\frac{2^p }{2C d k e^k (p+1)^2}}\le C 2^p \exp\left(-\frac{1}{C'} e^{\delta p/2}\right),
 \end{multline}
 {\red where in the last inequality we simply used the definition of $p^{(k)}_0$ (recall \eqref{pk0}).
Summing over $p\ge p_0$, we obtain that for $k$ sufficiently large
\begin{equation}
  \bbP\left(  \exists p\ge p_0, \ \maxtwo{x,y \in \bbD_{p}}{|x-y|\le \sqrt{d}2^{-p}} | X_k(x)-X_k(y)|\le \frac{1}{p(p+1)} \right) \le  e^{- k}.
\end{equation}
The field $Y_k$ possesses the same kind of regularity as $X_k$ so that the argument exposed above adapts verbatim to that case.}
\qed

\section{Proof of Lemma \ref{lemztim}}\label{tokz}
\noindent Let us start with the case $\gep,\gep'=0$ and prove 

\begin{equation}\label{topzz}
\left|K_n(x,y)- \min \left( n ,\log \frac{1}{|x-y|} \right)\right|\le C.
\end{equation}
The assumptions $Q_n(x,y)=0$ if $|x-y|\ge e^{-n}$ and $Q_n(x,y)\le \sqrt{Q(x,x)}\sqrt{Q(y,y)}=1$ immediately yields the upper bound. For the lower bound, we have, using the positivity and Lipshitz constant for $Q_k$ 
\begin{equation}
 K_n(x,y)\ge \sum_{k=1}^{ \min \left( n , \log \left(\frac{1}{|x-y|}\right) \right)} Q_k(x,y)
 \ge \sum_{k=1}^{ \min \left( n , \log \left(\frac{1}{|x-y|}\right) \right)} \left[Q_k(x,x)- C e^k|x-y|\right],
\end{equation}
and conclude from the fact that $\sum_{k=1}^{\log \left(\frac{1}{|x-y|}\right)} e^k|x-y|$ is bounded. 
From the definition of $Y_{n,\gep}$ in Equation \eqref{laconvolzz} we have 
\begin{equation}
 K_{n,\gep,\gep'}(x,y)= \int_{\bbR^d} K_n(z_1,z_2)\theta_{\gep}(z_1-x) \theta_{\gep'}(z_2-y)\dd z_1 \dd z_2.
\end{equation}
From \eqref{topzz} we can replace $K_n(z_1,z_2)$ by $\min \left( n ,\log \frac{1}{|z_1-z_2|} \right)$ and the results then follows from standard computations.

\qed

\bibliographystyle{plain}
\bibliography{../bibliography.bib}

\end{document}